\documentclass[13pt]{article}
\usepackage[top=1in, bottom=1in, left=1in, right=1in]{geometry}

\author{Adam Brown}
\title{Arakawa-Suzuki functors for Whittaker modules}
\date{}
\usepackage{amssymb}
\usepackage{amsmath}
\usepackage{graphicx}
\usepackage{manfnt}
\usepackage{amsthm}
\usepackage{mathrsfs}
\usepackage{mathtools}
\usepackage{epstopdf}
\usepackage{bigints}
\usepackage{color}
\usepackage{dsfont}
\usepackage{enumerate}
\usepackage{tikz}
\usetikzlibrary{arrows}
\usetikzlibrary{positioning}

\numberwithin{equation}{section}

\begin{document}

\maketitle

\newcommand{\bbm}{\begin{bmatrix}}
\newcommand{\ebm}{\end{bmatrix}}
\newcommand{\bv}{\begin{vmatrix}}
\newcommand{\ev}{\end{vmatrix}}
\newcommand{\no}{\noindent}
\newcommand{\g}{\mathfrak{g}}
\newcommand{\lev}{\mathfrak{l}}
\newcommand{\bb}{\mathfrak{b}}
\newcommand{\p}{\mathfrak{p}}
\newcommand{\td}{\mathfrak{t}}
\newcommand{\n}{\mathfrak{n}}
\newcommand{\h}{\mathfrak{h}}
\newcommand{\s}{\mathfrak{s}}
\newcommand{\z}{\mathfrak{z}}
\newcommand{\uu}{\mathfrak{u}}
\newcommand{\e}{\epsilon}
\newcommand{\N}{\mathcal{N}}
\newcommand{\HH}{\mathbb{H}}
\newcommand{\beq}{\begin{equation*}}
\newcommand{\eeq}{\end{equation*}}
\newcommand{\beqn}{\begin{eqnarray*}}
\newcommand{\eeqn}{\end{eqnarray*}}
\newcommand{\vs}{\vspace{10mm}}
\newcommand{\defeq}{\vcentcolon=}
\newcommand{\acts}{\circlearrowright}
\newcommand{\stdH}{\text{std}_\mathbb{H}}
\newcommand{\irrH}{\text{irr}_\mathbb{H}}
\newcommand{\stdN}{\text{std}_\mathcal{N}}
\newcommand{\irrN}{\text{irr}_\mathcal{N}}
\newcommand{\stdNg}{\text{std}_{\mathcal{N}_\g}}
\newcommand{\irrNg}{\text{irr}_{\mathcal{N}_\g}}
\theoremstyle{definition}
\newtheorem*{defi}{Definition}
\theoremstyle{plain}
\newtheorem{thm}{Theorem}[subsection]
\newtheorem*{thm*}{Theorem}

\newtheorem{lem}[thm]{Lemma}

\newtheorem{prop}[thm]{Proposition}
\theoremstyle{plain}
\newtheorem*{conj}{Conjecture}
\theoremstyle{remark}
\newtheorem{cor}[thm]{Corollary}
\theoremstyle{remark}
\newtheorem*{rem}{Remark}
\theoremstyle{remark}
\newtheorem*{ex}{Example}

\begin{abstract}
 In this paper we construct a family of exact functors from the category of Whittaker modules of the simple complex Lie algebra of type $A_n$ to the category of finite-dimensional modules of the graded affine Hecke algebra of type $A_\ell$. Using results of Backelin~\cite{Back} and of Arakawa-Suzuki~\cite{AS}, we prove that these functors map standard modules to standard modules (or zero) and simple modules to simple modules (or zero). Moreover, we show that each simple module of the graded affine Hecke algebra appears as the image of a simple Whittaker module. Since the Whittaker category contains the BGG category $\mathcal{O}$ as a full subcategory, our results generalize results of Arakawa-Suzuki~\cite{AS}, which in turn generalize Schur-Weyl duality between finite-dimensional representations of $\text{SL}_n(\mathbb{C})$ and representations of the symmetric group $S_n$.

\end{abstract}

\begin{section}{Introduction}
The geometric interpretation of coefficients of the Kazhdan-Lusztig polynomials in terms of dimensions of intersection homology groups on the flag variety~\cite{KL} was a breakthrough for the application of geometric techniques to the representation theory of semisimple complex Lie algebras. Lusztig later developed a generalization of the Springer correspondence which provided a geometric interpretation of finite-dimensional representations of the graded affine Hecke algebra~\cite{LLS2}. Surprisingly, these geometric theories can be related in a way which implies structural similarities between the irreducible characters of $\mathfrak{sl}_n(\mathbb{C})$ and of the graded affine Hecke algebra associated with the root system for $\mathfrak{sl}_n(\mathbb{C})$~\cite{Z}. Functorial interpretations of these geometric relationships have been explored by Arakawa-Suzuki~\cite{AS} and Suzuki~\cite{S} in the complex setting, and by Ciubotaru-Trapa~\cite{CTU},~\cite{CT} in the real setting. In this paper we explore a functorial interpretation of the geometric similarities between the Kazhdan-Lusztig theories for the category of Whittaker modules and graded affine Hecke algebra modules associated with $\mathfrak{sl}_n(\mathbb{C})$. In this way, we generalize the Schur-Weyl duality of Arakawa-Suzuki~\cite{AS} for category $\mathcal{O}$ to a Schur-Weyl duality for the category of Whittaker modules. 

We begin by summarizing our notation so that we can describe our main results. Returning temporarily to the setting of general complex semisimple Lie algebras, let $U(\g)$ be the universal enveloping algebra of a semisimple complex Lie algebra $\g$ and $Z(\g)$ be the center of $U(\g)$. Let $\Pi\subset \Delta^+$ be the set of simple and positive roots, respectively, corresponding to a choice of Cartan and Borel subalgebra $\h\subset\bb\subset \g$. Let $W$ be the Weyl group of $\g$ and $\n=[\bb,\bb]$ be the nilradical of $\bb$. For a root $\alpha\in\h^\ast$, let $\g_\alpha=\{x\in \g: \text{ad}(h)x=\alpha(h)x\text{ for all }h\in\h\}$, and let $\rho=\frac{1}{2}\sum_{\alpha\in\Delta^+}\alpha$ denote the half sum of positive roots. The category of Whittaker modules, denoted by $\N_\g$, consists of finitely generated $U(\g)$-modules which are locally $Z(\g)$-finite and locally $U(\n)$-finite. For $\g=\mathfrak{sl}_n(\mathbb{C})$, we define a family of exact functors from the category $\N_\g$ to the category of finite-dimensional modules for the graded affine Hecke algebra of the Coxeter system of type $A_\ell$. When $n=\ell$, we prove (under natural assumptions) that standard objects are mapped to standard objects (or zero) and irreducible objects are mapped to irreducible objects (or zero). The category $\mathcal{N}_\g$ contains the BGG category $\mathcal{O}$ as a full subcategory, and when we restrict to $\mathcal{O}$ we recover results of Arakawa-Suzuki~\cite{AS}. When we restrict to finite-dimensional $\mathfrak{sl}_n(\mathbb{C})$-modules, we recover the classical Schur-Weyl duality between finite-dimensional representations of $SL_n(\mathbb{C})$ and finite-dimensional representations of the symmetric group $S_n$. \par

We will now briefly review the notation needed to define the functor. Let $\eta\in\text{ch}\n:=(\n/[\n,\n])^\ast$ be a character of $\n$, $\Pi_\eta=\{\alpha\in\Pi:\eta\vert_{\g_\alpha}\neq 0\}$, and $W_\eta$ be the Weyl group generated by the reflections $s_\alpha$ for $\alpha\in \Pi_\eta$. Let $\p_\eta\subset\g$ be the corresponding parabolic subalgebra containing $\bb$ with ad$\h$-stable Levi decomposition $\p_\eta=\lev_\eta\oplus \n^\eta$. Let $\mathfrak{z}$ be the center of the reductive Lie algebra $\lev_\eta$ and $\s:=[\lev_\eta,\lev_\eta]$ be the semisimple part of $\lev_\eta$, so that $\lev_\eta=\z\oplus\s$. Set $\bb_\eta=\bb\cap\lev_\eta$ and let $\n_\eta$ be its nilradical. Let $Z(\lev_\eta)$ be the center of $U(\lev_\eta)$, and $\xi_\eta:\h^\ast\rightarrow\text{Max}Z(\lev_\eta)$ be induced by the relative Harish-Chandra homomorphism for $U(\lev_\eta)$. For a locally $Z(\g)$-finite $U(\g)$-module $X$, let $X^{[\lambda]}$ denote the subspace consisting of vectors with generalized $Z(\g)$-infinitesimal character corresponding to $\lambda\in\h^\ast$ via the Harish-Chandra homomorphism. Let $$\N_\g(\eta)=\{M\in\N_\g:\text{ $\forall$ $m\in M$ and $u\in\n$, $\exists$ $k\in\mathbb{Z}_{\ge 0}$ so that }(u-\eta(u))^km=0\}$$ be the full subcategory of Whittaker modules which are $U(\n)$-finite for the $\eta$-twisted action of $\n$. For a $U(\g)$-module $X\in \N_\g(\eta)$, let 
$$X_{\lambda_\z}=\{x\in X:\forall z\in\z,\quad (z-\lambda(z))^kx=0\}$$ denote the  generalized $\z$-weight space corresponding to $\lambda\in\h^\ast$ restricted to $\z\subset\h$.
Given a $U(\n_\eta)$-module $X$, define the degree zero $\eta$-twisted $\n_\eta$-cohomology of $X$ by 
$$H^0_\eta(\n_\eta,X):=\{x\in X :ux-\eta(u)x=0\quad\forall u\in\n_\eta\}.$$ Throughout the remainder of this section we will specialize to type $A$. Let $\g=\mathfrak{sl}_n(\mathbb{C})$ and $V=\mathbb{C}^n$ be the standard representation of $\g$. For a $U(\g)$-module $X\in\mathcal{N}_\g(\eta)$ we define the functor
\beq
F_{\ell,\eta,\lambda}(X):=H^0_\eta\left(\n_\eta,\left(X\otimes V^{\otimes \ell}\right)^{[\lambda]}_{\lambda_\mathfrak{z}}\right).
\eeq
Since $V$ is finite-dimensional and $X\in \N_\g(\eta)$, Theorem \ref{kt} implies that $X\otimes V^{\otimes \ell}\in\N_\g(\eta)$. Moreover, since $Z(\g)$ by definition commutes with $\z$, projection onto generalized $Z(\g)$-infinitesimal character commutes with projection onto a generalized $\z$-weight space. This shows that there is no ambiguity in the notation $\left(X\otimes V^{\otimes \ell}\right)^{[\lambda]}_{\lambda_\z}$. Additionally, Proposition \ref{ex_wt} shows that $\left( X\otimes V^{\otimes \ell}\right)^{[\lambda]}_{\lambda_\z}$ is an object in $\N_{\lev_\eta}(\eta)$. In particular, $\left( X\otimes V^{\otimes \ell}\right)^{[\lambda]}_{\lambda_\z}$ is a $U(\n_\eta)$-module with well defined $\eta$-twisted $\n_\eta$-cohomology. In Section \ref{hact}, following~\cite{AS}, we define an action of the graded affine Hecke algebra $\HH$ corresponding to the root datum for $SL_\ell$ on the $U(\g)$-module $X\otimes V^{\otimes \ell}$. The action of $\HH$ commutes with the action of $U(\g)$, and induces an $\HH$-module structure on $H^0_\eta\left(\n_\eta,\left(X\otimes V^{\otimes \ell}\right)^{[\lambda]}_{\lambda_\mathfrak{z}}\right)$. We can therefore view $F_{\ell,\eta,\lambda}$ as a functor from $\N_\g(\eta)$ to the category of $\HH$-modules. This family of functors possess a number of nice properties that we will explore in this paper and in future work. Our main result is the following (Theorem \ref{main}). 
\begin{thm*}
For $\lambda\in\h^\ast$ dominant, $F_{\ell,\eta,\lambda}$ is an exact functor from $\N_\g(\eta)$ to the category of finite-dimensional modules of the graded affine Hecke algebra corresponding to $\g$. Moreover, if $\lambda$ is a dominant integral weight such that $W_\eta=\{w\in W:w(\lambda+\rho)=\lambda+\rho\}$, and $X\in \N_\g(\eta)$ is irreducible with infinitesimal character corresponding to $\lambda$, then $F_{n,\eta,\lambda}(X)$ is irreducible or zero. 
\end{thm*}
We prove the above theorem by using results of Backelin~\cite{Back} to reduce the above statement to the corresponding result for category $\mathcal{O}$, which was proved by Arakawa-Suzuki~\cite{AS}. These functors provide representation theoretic interpretations of combinatorial relationships between double cosets of the symmetric group and multisegments (defined in Section \ref{multiseg-sec}) and geometric relationships between graded nilpotent classes and flag varieties (described in Section \ref{Hecke-geo-class} and in~\cite{Z}). Motivated by Suzuki's proof of Rogawski's conjecture~\cite{S}, we are interested in studying how the functors defined in this paper behave with respect to Jantzen filtrations of Whittaker modules. Additionally, it would be natural to extend our results to other classical Lie algebras. We plan to pursue these questions in future work. \par
The paper is organized as follows. In Section \ref{whit-sec}, we begin by reviewing the structure of Whittaker modules, and conclude with the multiplicities (written in terms of Kazhdan-Lusztig polynomials) of irreducible Whittaker modules in the composition series of standard Whittaker modules. In Section \ref{heck-sec}, we change directions and review the structure theory of graded affine Hecke algebras and their modules. Specifically, we review an algebraic construction of standard $\HH$-modules~\cite{E}, as well as a geometric construction~\cite{LLS2}, and a combinatorial parametrization for type $A_n$ following Zelevinsky~\cite{Z}. We conclude this section with a $p$-adic analogue of the Kazhdan-Lusztig conjectures, due independently to Lusztig and Zelevinsky~\cite{Z81}, which give multiplicity formulas for the standard graded affine Hecke algebra modules. Finally, in Section \ref{sec-main}, we study the relationships between Whittaker modules and $\HH$-modules. Initially, we construct an injective map from geometric parameters for irreducible $\HH$-modules to geometric parameters for irreducible Whittaker modules. We briefly recall the main properties of Arakawa-Suzuki functors for highest weight modules before concluding with the construction of Arakawa-Suzuki functors for Whittaker modules and our main theorem.

\noindent{\bf{Acknowledgments.}} The author would like to thank Peter Trapa and Dragan Mili\v{c}i\'c for their guidance in this project, which was completed while the author was a PhD student at the University of Utah under the supervision of Peter Trapa. 

\end{section}

\begin{section}{The category of Whittaker modules}\label{whit-sec}
We begin by reviewing the Mili\v ci\'c-Soergel~\cite{MS} construction of the category of Whittaker modules. Throughout this section, let $\g$ be a complex semisimple Lie algebra, with all of the relevant notation defined in the introduction. For $w\in W$ and $\lambda\in \h^\ast$ define the dot action of $W$ on $\h^\ast $ by 
$$
 w\bullet\lambda=w(\lambda+\rho)-\rho\quad\text{where}\quad\rho=\frac{1}{2}\sum_{\alpha\in \Delta^+}\alpha.
$$
 The Poincar\'e-Birkhoff-Witt basis theorem for $U(\g)$ relative to the triangular decomposition $\g=\bar{\n}\oplus\h\oplus\n$ gives us the vector space decomposition $U(\g)=U(\h)\oplus(\bar{\n}U(\g)+U(\g)\n)$. The Harish-Chandra homomorphism $\xi^\sharp:Z(\g)\rightarrow U(\h)$ is the projection map from $Z(\g)$ to $U(\h)$ by the above decomposition. This induces a map $\xi:\h^\ast\rightarrow \text{Max}Z(\g)$, where $\lambda\in\h^\ast$ maps to ker$\left(\lambda\circ\xi^\sharp\right)$. It is well known that $\xi(\lambda)=\xi(\mu)$ if and only if $ W\bullet\lambda=W\bullet\mu$. As in the introduction, suppose $\eta\in\text{ch}\n:=\left(\n/[\n,\n]\right)^\ast$ is a character of $\n$ with corresponding set of simple roots $\Pi_\eta=\{\alpha\in\Pi:\eta\vert_{\g_\alpha}\neq 0\}$ and Weyl group $W_\eta$  generated by the reflections $s_\alpha$ for $\alpha\in \Pi_\eta$. Let $\p_\eta\subset\g$ be the corresponding parabolic subalgebra containing $\bb$ with ad$\h$-stable Levi decomposition $\p_\eta=\lev_\eta\oplus \n^\eta$. Let $\bar{\n}^\eta$ be the orthocomplement (with respect to the Killing form) of $\p_\eta$ in $\g$, so that we have the decomposition $\g=\bar{\n}^\eta\oplus\p_\eta$. We will use $\mathfrak{z}$ to denote the center of $\lev_\eta$. Set $\bb_\eta=\bb\cap\lev_\eta$ and let $\n_\eta$ be its nilradical. Let $\xi^\sharp_\eta:Z(\lev_\eta)\rightarrow S(\h)$ be the Harish-Chandra homomorphism of $\lev_\eta$. That is, the projection $Z(\lev_\eta)$ to $ U(\h)$ induced by the decomposition $U(\lev_\eta)=U(\h)\oplus \left(\bar{\n}_\eta U(\lev_\eta)+U(\lev_\eta)\n_\eta\right)$ given by the Poincar\'e-Birkhoff-Witt basis theorem for $U(\lev_\eta)$. Let $\xi_\eta:\h^\ast\rightarrow \text{Max}Z(\lev_\eta)$ be the map induced from $\xi^\sharp_\eta$. We say that an $A$-module is locally finite if the $A$-span of any element of the module is finite-dimensional. 

We can now define the category of \emph{Whittaker modules}, denoted by $\N_\g$, to be the full subcategory of all $U(\g)$-modules which are finitely generated over $U(\g)$, locally finite over $U(\n)$, and locally finite over $Z(\g)$.\begin{rem}
We will drop the subscript $\g$ of $\N_\g$ when the context is clear. 
\end{rem}
 \begin{prop}{\label{ex_inf}}The action of $Z(\g)$ decomposes $\N$ into a direct sum 
\beq
\N=\bigoplus_{\chi\in\text{Max}Z(\g)}\N(\chi),
\eeq
where $\N(\chi)$ denotes modules in $\N$ that are annihilated by a power of $\chi$. This is to say that each object $X\in\N$ admits a decomposition
\beq
X=\bigoplus_{[\lambda]}X^{[\lambda]},
\eeq
where $X^{[\lambda]}\in \N(\xi(\lambda))$ and $\lambda$ ranges over coset representatives of $W\backslash\h^\ast$. Moreover, the functor from $\mathcal{N}$ to $\mathcal{N}(\xi(\lambda))$ which maps $X$ to $X^{[\lambda]}$ is exact.
\end{prop} 
The action of $\n$ gives us a decomposition
\beq
\N=\bigoplus_{\eta\in\text{ch}\n}\N(\eta),
\eeq
where $\N(\eta)=\{M\in\N:\text{ $\forall$ $m\in M$ and $u\in\n$, $\exists$ $k\in\mathbb{Z}_{\ge 0}$ so that }(u-\eta(u))^km=0\}$. Combining these decompositions, we have
\beq
\N=\bigoplus_{\chi,\eta}\N(\chi,\eta),
\eeq
where $N(\chi,\eta)=\N(\chi)\cap\N(\eta)$. For each choice $\chi\in\text{Max}Z(\g)$ and $\eta\in \text{ch}\n$, we see that the category $\N(\chi,\eta)$ contains the module $U(\g)/\chi U(\g)\otimes_{U(\n)}\mathbb{C}_\eta$. Therefore each category $\N(\chi,\eta)$ is nonzero. Moreover, each simple object $L$ in $\N$ is contained in $\N(\chi,\eta)$ for some choice $\chi$ and $\eta$ depending on $L$. 
\begin{prop}{\cite[Section 1]{MS}}{\label{se}}
The categories $\mathcal{N}(\chi,\eta)$ are closed under subquotients and extensions in the category of $U(\g)$-modules. 
\end{prop}

\begin{thm}{\cite{M},\cite[Theorem 2.5]{MS95}}
Every object $X\in\mathcal{N}$ has finite length. 
\end{thm}
\begin{proof}[Sketch of Proof]
While McDowell proved this theorem algebraically in \cite{M}, this fact follows easily from the geometric perspective of \cite{MS95}. Mili\v ci\'c and Soergel show that the Beilinson and Bernstein localization of an object $X\in\N$ with infinitesimal character $\xi(\lambda)$ is a holonomic $\mathcal{D}_\lambda$-module. Since we are assuming that each object in $\N$ is finitely generated and locally $Z(\g)$-finite, it follows that each object in $\N$ has finite length.
\end{proof}

\begin{subsection}{Classification of irreducible Whittaker modules}
In this section we will review the classification of standard and irreducible Whittaker modules \cite{M}. This algebraic classification generalizes the theory of Verma modules, and uses induction to define standard objects which have unique irreducible quotients. We will begin by defining the modules studied in \cite{K78} for the Lie algebra $\lev_\eta$. For any ideal $I\subset Z(\lev_\eta)$ define the $U(\lev_\eta)$-module
\beq
Y(I,\eta)=U(\lev_\eta)/IU(\lev_\eta)\otimes_{U(\n_\eta)}\mathbb{C}_\eta.
\eeq
Kostant showed that $Y(I,\eta)$ is irreducible for each $I\in \text{Max}Z(\lev_\eta)$. When $\eta=0$, $\lev_\eta=\h$ and $Y(I,\eta)$ is a one dimensional $\h$-module with weight $\lambda\in\h^\ast$ corresponding to the maximal ideal $I\in\text{Max}S(\h)$. Following the classical theory of generalized Verma modules, we define a $U(\g)$-module $\stdN(\lambda,\eta)$ by induction from $\lev_\eta$ to $\g$. Recall that $\xi_\eta(\lambda)\in \text{Max}Z(\lev_\eta)$ is induced from the relative Harish-Chandra homomorphism on $U(\lev_\eta)$. Define the standard Whittaker module corresponding to the pair $(\lambda,\eta)$ by
\beq
\stdN(\lambda,\eta):=U(\g)\otimes_{U(\p_\eta)}Y(\xi_\eta(\lambda),\eta),
\eeq
where we extend the $\lev_\eta$ action on $Y$ to an action of $\p_\eta$ by letting $\n^\eta$ act trivially. 
\begin{prop}{\cite[Proposition 2.1]{MS}}{\label{stdstruct}}
\beqn
&(1)& \text{$\stdN(\lambda,\eta)\cong \stdN(\mu,\eta)$ if and only if $W_\eta\bullet\lambda=W_\eta\bullet\mu$}\\
&(2)& \text{$\stdN(\lambda,\eta)$ has a unique simple quotient, denoted $\irrN(\lambda,\eta)$}\\
&(3)& \text{Ann$_{U(\g)}\stdN(\lambda,\eta)=\xi(\lambda)U(\g)$}
\eeqn

\end{prop}
\begin{prop}{\label{cs_k}}{\cite[Corollary 2.5]{MS}}
For all $\lambda\in\h^\ast$ the module $U(\g)/\xi(\lambda)U(\g)\otimes_{U(\n)}\mathbb{C}_\eta\in\N(\xi(\lambda),\eta)$ admits a filtration with subquotients isomorphic to $\stdN(w\bullet\lambda,\eta)$, where $w$ ranges of coset representatives of $W_\eta\backslash W$. 
\end{prop}
We can conclude from Proposition \ref{cs_k} and Proposition \ref{se} that $\stdN(\lambda,\eta)\in\N(\xi(\lambda),\eta)$. If $\eta=0$, then $\stdN(\lambda,0)$ is the usual Verma module with highest weight $\lambda$, and $\N(\xi(\lambda),\eta=0)$ contains the BGG category $\mathcal{O}_\lambda$ with generalized infinitesimal character $\xi(\lambda)$. 
\begin{rem}
When we want to emphasize the Lie algebra $\g$ for which we are considering Whittaker modules, we will use the notation $\stdNg(\lambda,\eta)=U(\g)\otimes_{\p_\eta}Y_{\lev_\eta}(\xi_\eta(\lambda),\eta)$. 
\end{rem}  
\begin{thm}{\cite[Proposition 2.4]{M}}{\label{mcd}} As $U(\lev_\eta)$-modules,
\beq
\stdN(\lambda,\eta)\cong U(\bar{\n}^\eta)\otimes_\mathbb{C}Y(\xi_\eta(\lambda),\eta).
\eeq
Moreover, the center of $\lev_\eta$ (denoted $\mathfrak{z}$) acts semisimply on $\stdN(\lambda,\eta)$, and the $\mathfrak{z}$-weight spaces $U(\bar{\n}^\eta)_{\gamma_\mathfrak{z}}$ are finite-dimensional $U(\lev_\eta)$-modules.  
\end{thm}
To avoid the unfortunate notational confusion between $\z$-weight spaces and $\h$-weight spaces, we will reluctantly use double subscripts and denote the generalized $\mathfrak{z}$-weight space of a module $X$ by $X_{\lambda_\z}$, for $\lambda\in\z^\ast$. The generalized $\h$-weight space corresponding to $\gamma\in\h^\ast$ will be denoted by the usual notation $X_\gamma$ (where $\gamma$ has no subscript). The following theorem appears as Theorem 4.6 in \cite{K78}, and will be crucial to our understanding of the structure of Whittaker modules. 
\begin{thm}{\cite[Theorem 4.6]{K78}}{\label{kt}}
Let $F$ be a finite-dimensional $U(\g)$-module. Assume $\eta$ is nondegenerate ($\lev_\eta=\g$), and let $Y$ be an irreducible Whittaker module in $\N(\xi(\lambda),\eta)$. Let $T=F\otimes Y$ be the tensor product $U(\g)$-module. Then $T$ is an object in $\N(\eta)$ and composition series length equal to dim$ F$. In particular, 
\beq
\text{dim} H^0_\eta(\n,T)=\text{dim} F.
\eeq
Suppose dim $F=k$. Then there is a composition series $0=T_0\subset T_1\subset\cdots\subset T_k=T$ such that $T_i/T_{i-1}\cong Y(\xi(\lambda+\nu_i),\eta)$, where $\nu_i$ are the weights of $F$ (counting multiplicities) ordered corresponding to a $\bar{\bb}$-invariant flag of $F$.

\end{thm}
Observe that the above filtration is induced by a $\bar{\bb}$-invariant flag of $F$, rather than a $\bb$-invariant flag, as would be the case for category $\mathcal{O}$. 
\begin{cor}{\cite[Lemma 5.12]{MS}}{\label{dwt}}
Let $F$ be a finite-dimensional $U(\g)$-module. Let $\eta$ be any character of $\n$. Let $T=F\otimes \stdN(\lambda,\eta)$ for some $\lambda\in\h^\ast$. Then $T$ has a filtration with subquotients $\stdN(\lambda+\nu,\eta)$ for weights $\nu$ of $F$ (counting multiplicity). 
\end{cor}
\begin{proof}
We begin with the definition of standard Whittaker modules 
\beq
T=F\otimes \stdN(\lambda,\eta)=F\otimes\big(U(\g)\otimes_{\p_\eta}Y_{\lev_\eta}(\xi_\eta(\lambda),\eta)\big).
\eeq 
The Mackey isomorphism gives us the isomorphism of $U(\g)$-modules
\beq
T\cong U(\g)\otimes_{\p_\eta}\big(F\otimes Y_{\lev_\eta}(\xi_\eta(\lambda),\eta)\big).
\eeq
Now we can apply Theorem \ref{kt} to $F\otimes Y_{\lev_\eta}(\xi_\eta(\lambda),\eta)$. Since parabolic induction is an exact functor, it descends to the level of Grothendieck groups. Therefore the composition factors of $T$ are of the form $U(\g)\otimes_{\p_\eta}Y_{\lev_\eta}(\xi_\eta(\lambda+\nu),\eta)=\stdN(\lambda+\nu,\eta)$ for weights $\nu$ of $F$ (counting multiplicity). 
\end{proof}
The following theorem concludes the classification of simple objects in $\N$. 
\begin{thm}{\label{irrN_class}}{\cite{M},\cite[Theorem 2.6]{MS}}
Each simple $U(\g)$-module contained in $\N$ is isomorphic to $\irrN(\lambda,\eta)$ for some choice of $\eta\in\text{ch}\n$ and $\lambda\in \h^\ast$. 

\end{thm}
\begin{proof}
Each simple object $L$ in $\N$ has an infinitesimal character and is contained in $\N(\xi(\lambda),\eta)$ for some $\lambda\in\h^\ast $ and $\eta\in \text{ch}\n$. Therefore $L$ is a subquotient of $U(\g)/\xi(\lambda)U(\g)\otimes_{U(\n)}\mathbb{C}_\eta$. Proposition \ref{cs_k} implies that $L$ is a subquotient of $\stdN(w\bullet\lambda,\eta)$ for some $w\in W$. Any simple subquotient of $\stdN(w\bullet\lambda,\eta)$ must have a $\z$-weight $\mu\in\z^\ast$ so that $L_{\mu_\z}\neq 0$ and $\n^\eta L_{\mu_\z}=0$. There exists $I\in \text{Max }Z(\lev_\eta)$ so that $\text{Hom}_{\lev_\eta}(Y(I,\eta),L_{\mu_\z})\neq 0$. Therefore $\text{Hom}_\g(\stdN(\nu,\eta),L)\neq 0$ for $\nu\in\h^\ast$ such that $\xi_\eta(\nu)=I$. This proves that $L\cong \irrN(\nu,\eta)$ since $L$ is simple. 
\end{proof}
\begin{rem}
It follows from Proposition \ref{stdstruct} and Theorem \ref{irrN_class} that $W_\eta\backslash W\slash W_\lambda$ parametrizes the irreducible objects of $\mathcal{N}(\xi(\lambda),\eta)$.
\end{rem}  

\end{subsection}
\begin{subsection}{Whittaker vectors and $\z$-weight vectors of Whittaker modules}
 Let $X^{\n}=\{x\in X| ux-\eta(u)x=0\text{ for all $u\in \n$}\}$ be the subspace of $\n$-invariant vectors for the $\eta$-twisted action of $\n$ on $X$. Let $H^\bullet_\eta(\n,X)$ denote $\eta$-twisted $\n$-Lie algebra cohomology of $X$, i.e. the right derived functor of the $\n$-invariants functor $X\mapsto X^{\n}$ (for the $\eta$-twisted action of $\n$ on $X$). An  \emph{$\eta$-Whittaker vector} of $X\in \mathcal{N}$ is a vector $v\in X$ so that $nv=\eta(n)v$ for all $n\in \n$. Equivalently, an $\eta$-Whittaker vector of $X\in\N$ is a vector contained in $H^0_\eta(\n, X)$.

\begin{lem}{\cite[Lemma 5.8]{MS95}}{\label{exact}}
For $\g$ semisimple and $\eta$ nondegenerate ($\lev_\eta=\g$), the functor $H^0_\eta(\n,\cdot)$ from the category $\mathcal{N}(\eta)$ to the category of $Z(\g)$-modules is exact. 
\end{lem}
\begin{proof}
First we prove that $H_\eta^i(\n,V)=0$ when $i\ge 1$ and $V$ is an irreducible Whittaker module. Let $N$ be a connected algebraic group with Lie algebra $\n$ and a morphism of $N$ into the group of inner automorphisms of $\g$ such that its differential induces an injection of $\n$ into $\g$. In \cite{MS95}, Mili\v{c}i\'{c} and Soergel show that an irreducible Whittaker module viewed as a $U(\n)$-module with $\eta$-twisted action of $\n$, for nondegenerate $\eta$, is isomorphic to the differential of the natural action of $N$ on the algebra of regular functions $R(C(w_0))$ on the open cell $C(w_0)$ of the flag variety of $\g$. Under this isomorphism, the Whittaker vector generating the irreducible Whittaker module is mapped to the spherical vector $\mathbf{1}\in R(C(w_0))$. Notice that $C(w_0)$ is a single $N$ orbit on the flag variety of $\g$. Since, $N$ and $C(w_0)$ are connected, $N$ acts transitively on $C(w_0)$, and $\text{dim}N=\text{dim}C(w_0)$, we can conclude that $N$ and $C(w_0)$ are isomorphic as varieties. We can therefore consider the corresponding action of $\n$ on the algebra of regular functions on $N$. We will proceed by induction on the dimension of $\n$, following the proof of Lemma 1.9 in \cite{MilicicPandzic}. If $\text{dim}\n=1$, then $\n$ is abelian. Since $N$ is an affine space, it follows that $R(N)$ is a polynomial algebra, and  $H^i(\n,R(N))$ is the cohomology of the Koszul complex with coefficients in $R(N)$. By the Poincar\'{e} lemma, we have that $H^i(\n,R(N))=0$ for $i\ge 1$. If $\text{dim}\n>1$, we consider the commutator subgroup $N'=(N,N)$, with Lie algebra $\n'\subsetneq\n$. By the induction hypothesis, the Hochschild-Serre spectral sequence of Lie algebra cohomology collapses, giving us the equality 
$$
H^p(\n/\n', R(N/N'))=H^p(\n, R(N)).
$$
for $p\ge 0$. Since $\n/\n'$ is abelian, we can conclude (by the first part of the proof) that $H^p(\n/\n', R(N/N'))$ is zero for $p>0$. Therefore $H^i_\eta(\n,V)=0$ for all $i\ge 1 $ and irreducible Whittaker modules $V$. To complete the proof we will now consider an arbitrary Whittaker module $V$. Since $V$ has finite length, we will proceed by induction on the length of $V$. Suppose $V$ has length two, and consider the short exact sequence 
\beq
0\rightarrow S\rightarrow V\rightarrow I\rightarrow 0,
\eeq
where $S$ is an irreducible submodule and $I$ is an irreducible quotient. The long exact sequence of Lie algebra cohomology shows that $H^i(\n,V)=H^i(\n,I)$ for $i\ge 1$. By induction on the length of $V$ we see that $H^i(\n,V)=0$ for $i\ge 1$. 
\end{proof}
Recall the Cartan decomposition $\lev_\eta=\bar{\n}_\eta\oplus\h\oplus\n_\eta$ for the Levi subalgebra $\lev_\eta$, and the decomposition of $\g$ given by considering the subalgebra $\p_\eta$ and its orthocomplement $\bar{\n}^\eta$ (with respect to the Killing form) 
$$\g=\bar{\n}^\eta\oplus\p_\eta=\bar{\n}^\eta\oplus\lev_\eta\oplus \n^\eta.$$ Let $\s:=[\lev_\eta,\lev_\eta]$ be the semisimple part of the Levi subalgebra $\lev_\eta$.
\begin{prop}{\label{ex_wt}}
For any $\gamma\in\z^\ast$, the functor
\beqn
\mathcal{N}_\g(\eta)&\rightarrow&\mathcal{N}_\s(\eta_\s)\\
X&\mapsto& X_{\gamma_\z}:=\{x\in X:\forall z\in\z\quad (z-\gamma(z))^kx=0\}
\eeqn
is exact. Here we use the notation $\eta_\s$ to denote the restriction of $\eta$ (viewed as a function on $\n$) to $\n_\eta=\n\cap\s$. 
\end{prop}
\begin{proof}
We will begin by showing that the functor described above is well defined. Since any object $X\in \N_\g(\eta)$ has finite length and irreducible objects, being quotients of standard modules, are semi-simple over $\z$, it follows that $X$ is locally $U(\z)$-finite. 

Since $\z$ is a nilpotent Lie algebra that acts locally finitely on modules in $\mathcal{N}_\g(\eta)$, given an exact sequence of Whittaker modules $0\rightarrow A\rightarrow B\rightarrow C\rightarrow 0$, we get an exact sequence of $U(\z)$-modules by taking generalized weight spaces
\beq
0\rightarrow A_{\gamma_\z}\rightarrow B_{\gamma_\z}\rightarrow C_{\gamma_\z}\rightarrow 0.
\eeq
Here the morphisms are obtained by restricting the maps in the above exact sequence to the subspaces of generalized weight vectors.

We now turn our attention to showing that these generalized weight spaces are in fact Whittaker modules for the semisimple part of the Levi subalgebra. Again, for $X\in\mathcal{N}_\g(\eta)$, consider a Jordan-Holder filtration of $X$
\beq
0=V_0\subset V_1\subset\cdots\subset V_N=X,
\eeq
where $V_i/V_{i-1}$ is irreducible. We will show $X_{\gamma_\z}\in\mathcal{N}_\s(\eta_\s)$ by induction on the length of $X$. Assume $N=1$. Then $X$ is irreducible. We aim to show $X_{\gamma_\mathfrak{z}}\in\mathcal{N}_\mathfrak{s}(\eta_\s)$. There exists $\lambda$ so that the unique irreducible quotient of $\stdN(\lambda,\eta)$ is isomorphic to $X$. Theorem \ref{mcd} implies that $\stdN(\lambda,\eta)_{\gamma_\mathfrak{z}}\cong F \otimes Y_\mathfrak{s}(\xi_\mathfrak{s}(\lambda),\eta_\s)$ as $U(\mathfrak{s})$-modules, where $F$ is a finite-dimensional $U(\s)$-module. Since $F$ is finite-dimensional, we can apply Theorem 4.6 of \cite{K78} and conclude that $\stdN(\lambda,\eta)_{\gamma_\mathfrak{z}}$ has finite length as a $U(\s)$-module, and has compositions factors contained in $\mathcal{N}_\s(\eta_\s)$. Since $\mathcal{N}_\s(\eta_\s)$ is closed under extensions and $\stdN(\lambda,\eta)_{\gamma_\mathfrak{z}}$ has finite length, we can further conclude that $\stdN(\lambda,\eta)_{\gamma_\mathfrak{z}}$ is an object in $\mathcal{N}_\s(\eta_\s)$. Since $\mathfrak{z}$ acts semisimply on $\stdN(\lambda,\eta)$, we get the following exact sequence of $U(\mathfrak{s})$-modules
\beq
\stdN(\lambda,\eta)_{\gamma_\mathfrak{z}}\rightarrow X_{\gamma_\mathfrak{z}}\rightarrow 0.
\eeq
Since $\mathcal{N}_\mathfrak{s}(\eta_\s)$ is closed under quotients, we can conclude that $X_{\gamma_\mathfrak{z}}\in\mathcal{N}_\mathfrak{s}(\eta_\s)$. Now we will proceed with the inductive step. Assume that $(V_{n-1})_{\gamma_\z}\in \mathcal{N}_\s(\eta_\s)$. Then 
\beq
0\rightarrow (V_{n-1})_{\gamma_\z}\rightarrow (V_{n})_{\gamma_\z} \rightarrow (V_n/V_{n-1})_{\gamma_\z} \rightarrow 0
\eeq
is an exact sequence of $U(\s)$-modules. Since $\mathcal{N}_\s(\eta_\s)$ is closed under extensions, and $V_n/V_{n-1}$ is irreducible, we can conclude that $(V_{n})_{\gamma_\z}\in \mathcal{N}_\s(\eta_\s)$. In summary, we have the following exact sequence of $U(\z)$-modules
\beq
0\rightarrow A_{\gamma_\z}\rightarrow B_{\gamma_\z}\rightarrow C_{\gamma_\z}\rightarrow 0,
\eeq
where the morphisms are restrictions of $U(\g)$-module morphisms, and the objects are naturally $U(\s)$-modules contained in $\mathcal{N}_\s(\eta_\s)$. Therefore this is an exact sequence in $\mathcal{N}_\s(\eta_\s)$.
\end{proof}
\end{subsection}

\begin{subsection}{Functors between highest weight modules and Whittaker modules}

The following functor developed by Backelin in \cite{Back} will be crucial to our proof that standard objects in the Whittaker category are mapped (by the functor described in the introduction) to standard objects (or zero) in the category of Hecke algebra modules. \par 
The following construction is inspired by ideas in \cite{K78} and developed in \cite{Back}. For $\gamma\in\h^\ast$, let $X_{\gamma}$ denote the generalized $\gamma$-weight space of $X$. Let $P(V)$ denote the set of nonzero generalized weights of a $U(\g)$-module $V$.
 \begin{defi}
Let $X$ be a $U(\g)$-module contained in $\mathcal{O}':=\N(\eta=0)$. The \emph{completion module} of $X$, denoted $\overline{X}$, is the direct product
 \beq
 \overline{X}=\prod_{\gamma\in P(X)}X_{\gamma}.
 \eeq
 An element of $\overline{X}$ is a formal infinite sum
 \beq
 v=\sum_{\gamma\in P(X)}v_{\gamma}\in  \overline{X},
 \eeq
  where $v_\gamma\in X_\gamma$. We can extend the $U(\g)$-module structure of $X$ to the completion by defining
 \beq
 u_\nu v:=\sum_{\gamma\in P(X)}u_\nu v_{\gamma-\nu},
 \eeq
 where $u_\nu\in U(\g)_\nu$ (the $\nu$-weight space of $U(\g)$ viewed as a $\g$ module with that ad action of $\g$), and $u_\nu v_{\gamma-\nu}=0$ if $\gamma-\nu\notin P(X)$.  
 \end{defi}

 \begin{lem}{\label{tensor_prod}}
 Let $F$ be a finite-dimensional $U(\g)$-module, and $X$ be a $U(\g)$-module contained in $\mathcal{O}'$. Then $F\otimes_\mathbb{C}\overline{X}=\overline{F\otimes_\mathbb{C} X}$.
 \end{lem}
 \begin{proof}
 Any element of $\overline{F\otimes X}$ is a formal infinite sum $\sum t_\tau$, where $t_\tau$ is a weight vector of $F\otimes X$ with weight $\tau$. Each weight vector $t_\tau$ can be written as a sum of simple tensors of weight vectors of $F$ and $X$,
 \beq
 t_\tau=\sum_{\alpha+\beta=\tau} f_\alpha\otimes x_\beta.
 \eeq
Therefore, we can write elements of $\overline{F\otimes X}$ as $\sum_J f_\alpha\otimes x_\beta$, where $J=P(F)\times P(X)$, $f_\alpha\in F$ is a weight vector of $F$, and $x_\beta\in X$ is a weight vector of $X$. Let $\{v_i:i\in I\}$ be a basis of $F$. We can write each element $f_\alpha$ as $f_\alpha=\sum_I a_{i}(\alpha)v_i$, where $a_{i}(\alpha)\in\mathbb{C}$. Therefore, 
 \beq
 \sum_J f_\alpha\otimes x_\beta=\sum_J\sum_I a_{i}(\alpha)v_i\otimes x_\beta=\sum_I\sum_J a_{i}(\alpha)v_i\otimes x_\beta=\sum_I v_i\otimes\bigg(\sum_{J} a_{i}(\alpha)x_\beta\bigg).
 \eeq
 We can rewrite the sum $\sum_J a_{i}(\alpha)x_\beta$ as a formal infinite sum of generalized weight vectors of $X$ (for each $i$), since each $x_\beta$ is already a generalized weight vector, and since there are only finitely many $(\alpha,\beta)\in J$ for a fixed $\beta$. This shows that $\sum_I v_i\otimes\big(\sum_J a_{i}(\alpha)x_\beta\big)\in F\otimes\overline{X}$ since $I$ is a finite set. Therefore $\overline{F\otimes X}\subseteq F\otimes\overline{X}$. The reverse inclusion is obvious. 
 \end{proof}
 \begin{defi}
 We will also define the subspace of $\eta$-finite vectors
\beq
X_\eta=\{x\in X: U_\eta(\n)^kx=0\quad \text{for some } k\in\mathbb{Z}_{\ge0}\},
\eeq
where $U_\eta(\n)$ is the kernel of $\eta:U(\n)\rightarrow \mathbb{C}$. 
 \end{defi}
\begin{defi}
Let $\overline{\Gamma}_\eta$ be the functor from the highest weight category $\mathcal{O}'$ to the category of Whittaker modules $\N(\eta)$ defined by
\beqn
\overline{\Gamma}_\eta:\mathcal{O}'&\rightarrow &\N(\eta)\\
X&\mapsto& (\overline{X})_\eta
\eeqn
\end{defi}
\begin{thm}{\label{back}}{\cite[Proposition 6.9]{Back}}
$\overline{\Gamma}_\eta$ is an exact functor with the following properties 
\begin{enumerate}
\item $\overline{\Gamma}_\eta(M(\lambda))=\stdN(\lambda,\eta)$ for all $\lambda\in \h^\ast$.
\item $\overline{\Gamma}_\eta(L(\lambda))=\irrN(\lambda,\eta)$ if $\lambda$ is $\n_\eta$-antidominant.
\item $\overline{\Gamma}_\eta(L(\lambda))=0$ if $\lambda$ is not $\n_\eta$-antidominant.
\end{enumerate}
\end{thm}
 The following lemma will be crucial to the proof of Theorem \ref{back}. 
 \begin{lem}{\cite[Lemma 6.5]{Back}}
 For each Verma module $M(\lambda)$, 
\beq
\overline{\Gamma}_\eta(M(\lambda))=U(\g)\cdot H^0_\eta(\n,\overline{\Gamma}_\eta(M(\lambda)))
\eeq 
is an equality of $U(\g)$-modules.
 \end{lem}
 \begin{proof}
 The proof follows directly from Theorem 3.8, Lemma 3.9, and Theorem 4.4 of \cite{K78}.
 \end{proof}

 \begin{prop}{\label{ten_com}}
 Let $F$ be a finite-dimensional $U(\g)$-module and $\lambda\in\h^\ast$. Then
 \beq
 \overline{\Gamma}_\eta(M(\lambda)\otimes F)= \overline{\Gamma}_\eta(M(\lambda))\otimes F.
 \eeq
 \end{prop}
\begin{proof}
By Lemma \ref{tensor_prod}, we have that $\overline{M(\lambda)\otimes F}=\overline{M(\lambda)}\otimes F$. If $v\in \overline{M(\lambda)}_\eta$, then $U_\eta(\n)^k(v\otimes f)=0$ for any $f\in F$ and some $k
>>0$. Therefore, we have the inclusion $\overline{M(\lambda)}_\eta\otimes F\subseteq \overline{M(\lambda)\otimes F}_\eta$. From Corollary 2.1.5 and Theorem 2.4.2, we know that $\overline{M(\lambda)}_\eta\otimes F $ has a standard filtration with quotients $\stdN(\lambda+\tau,\eta)$ for each $\tau\in P(F)$ (including multiplicity). Similarly, we have that $M(\lambda)\otimes F$ has a standard filtration with quotients $M(\lambda+\tau)$ for $\tau\in P(F)$. Since $\overline{\Gamma}_\eta$ is exact, and maps $M(\lambda+\tau)$ to $\stdN(\lambda+\tau,\eta)$, we can conclude that $\overline{M(\lambda)}_\eta\otimes F$ and $\overline{M(\lambda)\otimes F}_\eta$ have isomorphic standard filtrations.
\end{proof}
\end{subsection}

\begin{subsection}{Composition series of Whittaker modules}
We can now use Theorem \ref{back} to calculate the multiplicity of irreducible Whittaker modules in the composition series of standard Whittaker modules.

\begin{thm}{\cite{MS},\cite[Theorem 6.2]{Back}}{\label{whitklp}}
Assume $\lambda\in\h^\ast$ is dominant. In the Grothendieck group of $\mathcal{N}(\xi(\lambda),\eta)$, we have
\beqn
[\stdN(w\bullet\lambda,\eta)]=\sum_{y\in I} P_{w,y}(1)[\irrN(y\bullet \lambda, \eta)],
\eeqn
where the sum is taken over $I=\{y\ge w : y\text{ is the longest element of }W_\eta y W_\lambda\}$ and $P_{w,y}$ are the Kazhdan-Lusztig polynomials of the Coxeter system $(W,S_\Pi)$. 
\end{thm}

\begin{cor}{\label{whit_geo_mult}}
We can give an alternative description of the multiplicity formulas in terms of the geometry of the flag variety. Let $C(w)$ be the Bruhat cell of the flag variety corresponding to $w\in W$. For $x\in C(w)\subset \overline{C({}_\eta y_\lambda)}$,
\beq
[ \stdN(w\bullet\lambda,\eta):\irrN(y\bullet\lambda,\eta)]=\sum_i\text{dim}H^i(\text{IC}_{x}(C({}_\eta y_\lambda))).
\eeq
Here $\text{IC}_x(C({}_\eta y_\lambda))$ denotes the stalk at $x$ of the intersection cohomology complex corresponding to the trivial local system on $C({}_\eta y_\lambda)$. 
\end{cor}
\end{subsection}

\end{section}

\begin{section}{Graded affine Hecke algebra modules}\label{heck-sec}
Hecke algebras appear naturally in the representation theory of algebraic groups over $p$-adic fields. Most notably, affine Hecke algebras were used in the proof of the Deligne-Langlands conjecture for irreducible representations of $p$-adic groups in~\cite{KL87}. By introducing a filtration on the affine Hecke algebra, Lusztig constructs a corresponding graded algebra, whose representation theory is closely related that of the affine Hecke algebra~\cite{L}. The representation theory of the graded affine Hecke algebra is in many ways easier to study, and the relationship between graded affine Hecke algebras and affine Hecke algebras can be thought of as analogous to the relationship between Lie algebras and Lie groups. Specifically, the graded affine Hecke algebra can be studied using methods of equivariant K-homology. With these tools available, Lusztig was able to construct standard and irreducible modules for the Hecke algebra, as well as compute the composition series of standard modules in terms of intersection homology~\cite{LLS2}. \par 
  In this section we will review an algebraic construction of standard and irreducible graded affine Hecke algebras due to Evens~\cite{E}, as well as the corresponding composition series in terms of the geometric parametrization of standard modules due to Lusztig. Finally, we will discuss a useful combinatorial parametrization due to Zelevinsky~\cite{Z} in the case where $G$ is of type $A_n$.

\begin{subsection}{Graded affine Hecke algebras}
We will now define the graded affine Hecke algebra introduced by Lusztig \cite{L}. Let $(X,R,   Y, R^\vee,\Pi)$ be a based root datum, with $V^\ast=\mathbb{C}\otimes_\mathbb{Z}X$ and $V=\mathbb{C}\otimes_\mathbb{Z}Y$. Let $W$ be the reflection group generated by simple reflections $s_\alpha$ for $\alpha\in\Pi$.
\begin{defi}
The graded affine Hecke algebra $\HH$ of the based root datum $(X,R,   Y, R^\vee,\Pi)$ is the unital associative algebra over $\mathbb{C}$ generated by $\{t_w:w\in W\}$ and $\{t_h:h\in V\}$, subject to the relations
\begin{enumerate}[i.]
\item The map $w\mapsto t_w$ from $\mathbb{C}[W]$ to $\HH$ is an algebra injection.
\item The map $h\mapsto t_h$ from $S(V)$ to $\HH$ is an algebra injection.
\item The generators satisfy the following commutation relation
\beq
t_{s_\alpha} t_h-t_{s_\alpha(h)}t_{s_\alpha}=\langle \alpha,h\rangle  \text{ for all $\alpha\in \Pi$ and $h\in V$}.
\eeq
\begin{rem}
The map $w\otimes h\mapsto t_wt_h$ defines a vector space isomorphism from $\mathbb{C}[W]\otimes S(V)$ to $\HH$. 
\end{rem}
For notational convenience we will write $w$ instead of $t_w$ and $h$ instead of $t_h$. Rewriting the commutation relations in condition iii, we get
\beq
s_\alpha h-s_\alpha(h) s_\alpha=\langle \alpha,h\rangle \text{ for all $\alpha\in \Pi$ and $h\in V$}.
\eeq
\end{enumerate}

\end{defi}

Let $\mathfrak{a}^\ast=\{x\in V^\ast:\alpha^\vee(x)=0\quad\forall \alpha\in\Pi\}$ and $\mathfrak{a}=\{x\in V:\lambda(\alpha)=0\quad\forall\alpha\in\Pi\}$. We will define the based root datum $(X_{ss},R, Y_{ss},R^\vee, \Pi)$ by considering the subsets of $X$ and $Y$ which are perpendicular to $\mathfrak{a}$ and $\mathfrak{a}^\ast$ respectively. Let $X_{ss}=\{x\in X: x(a)=0\quad\forall a\in\mathfrak{a}\}$ and $Y_{ss}=\{y\in Y: a'(y)=0\quad\forall a'\in\mathfrak{a}^\ast\}$. Then $S(\mathfrak{a})$ is in the center of $\HH$ and we have the decomposition (as algebras)
\beq
\HH\cong \HH_{ss}\otimes S(\mathfrak{a}),
\eeq
where $\HH_{ss}$ is the graded affine Hecke algebra associated to the root datum $(X_{ss},R, Y_{ss},R^\vee, \Pi)$. 
 \begin{lem}{\cite[Lemma 4.5]{L}}
The center $Z(\HH)$ of $\HH$ is 
\beq
Z(\HH)=S(V)^{W}.
\eeq
Moreover, central characters of $\HH$ (and maximal ideals of $S(V)^W$) are parametrized by $W$ orbits of $\lambda\in V^\ast$ (with the usual action $w\lambda$ of $W$ on $V^\ast$). Let $\chi_\lambda$ denote the maximal ideal in $S(V)^W$ corresponding to $\lambda\in V^\ast$.
\end{lem}
 Let $\mathcal{H}$ denote the category of finite-dimensional $\HH$-modules and $\mathcal{H}_\lambda$ denote the subcategory of finite-dimensional $\HH$-modules with central character corresponding to the maximal ideal $\chi_\lambda$. 
 
\end{subsection}
\begin{subsection}{Classification of finite-dimensional irreducible graded affine Hecke algebra modules}
In this section we will review three types of classifications of simple $\HH$-modules. First we will review the geometric construction of standard $\HH$-modules, following \cite{LLS2}. These $\HH$-modules have unique irreducible quotients, and parametrize isomorphism classes of simple $\HH$-modules. We will then review the algebraic approach of \cite{E}, parametrizing simple modules by pairs $(\HH_\p,U)$, where $\HH_\p$ is a parabolic subalgebra of $\HH$ and $U$ is a tempered representation of $\HH_\p$. Finally, we will consider the combinatorial parametrization developed in \cite{BZ} for the case when $\HH$ corresponds to root datum of type $A_n$.

\begin{subsubsection}{Geometric classification}{\label{Hecke-geo-class}}
Let $G$ be the connected complex reductive group with root datum $(X,R,Y,R^\vee,\Pi)$, with Lie algebra $\g$, and flag variety $\mathcal{B}$. Let $\sigma\in \g$ be semisimple, let $  L_\sigma=\{g\in  G:\text{Ad}(g)(\sigma)=\sigma\}$, and let $\mathcal{N}:=\{x\in \g:\text{ad}(x)\text{ is nilpotent}\}$ denote the nilpotent cone in $\g$. Consider the Springer resolution of the nilpotent cone
\beqn
\widetilde{\mathcal{N}}:=\{(x,\bb)\in \mathcal{N}\times \mathcal{B}:x\in\bb\}&\xrightarrow{\mu}& \mathcal{N}\\
(x,\bb)&\mapsto&x
\eeqn
Let $\mathcal{N}^\sigma$ denote the subvariety of $\mathcal{N}$ fixed by ad$(\sigma)$, $\widetilde{\mathcal{N}}^\sigma=\{(x,\bb)\in\N^\sigma\times \mathcal{B}:x,\sigma\in\bb\}$, and $\mu^\sigma$ denote the restriction of $\mu$ to $\widetilde{\mathcal{N}}^\sigma$. Let $\mathcal{C}_{\widetilde{\mathcal{N}}^\sigma}$ denote the constant perverse sheaf on $\widetilde{\mathcal{N}}^\sigma$. Let $$\g_{1}(\sigma):=\{x\in \g:\text{ad}(\sigma)(x)=x\},$$ and $\mathcal{P}_\HH(\sigma)$ denote the set of pairs $(\mathcal{O},\mathcal{E})$ such that $\mathcal{O}$ is an $ L_{\sigma}$ orbit in $ \g_{1}(\sigma)$ and $\mathcal{E}$ is an $L_\sigma$-equivariant local system on $\mathcal{O}$ such that $\text{IC}(\mathcal{O},\mathcal{E})$ appears in the decomposition of $\mu^\sigma_\ast\mathcal{C}_{\widetilde{\mathcal{N}}^\sigma}$. In \cite{LLS2} (cf.~\cite{ChrissGinzburg}), Lusztig constructs an action of $\HH$ on the vector space $\mathcal{H}^\bullet(i^!_x(\mu^\sigma_\ast(\mathcal{C}_{\widetilde{\mathcal{N}}^\sigma})))_\chi$, where $i_x:\{x\}\hookrightarrow \g_1(\sigma)$, and $\chi$ is a representation of the component group (corresponding to the local system $\mathcal{E}$) of $Z_{ G}(\sigma, x)$, the simultaneous centralizer of $\sigma$ and $x$ in $G$. Note that since $\mu^\sigma$ is a proper map, we can identify the above vector space with the homology of the fiber of $\mu^\sigma$ at $x$, as is done in~\cite{ChrissGinzburg}. The $\HH$-module $\mathcal{H}^\bullet(i^!_x(\mu^\sigma_\ast(\mathcal{C}_{\widetilde{\mathcal{N}}^\sigma})))_\chi$ only depends on the $ L_\sigma$ orbit of $x$ and the local system $\mathcal{E}$. Therefore, we will denote this module by $\stdH(\mathcal{O},\mathcal{E})$, and refer to it as a standard $\HH$-module corresponding to the geometric parameters $(\mathcal{O},\mathcal{E})\in \mathcal{P}_\HH(\sigma)$. 
 \begin{rem}
 If $\lambda\in V^\ast$ corresponds to $\sigma\in V$ under the trace form (with $V$ identified with $\h$), then $\stdH(\mathcal{O},\mathcal{E})$ will have central character $\chi_\lambda$ for each $(\mathcal{O},\mathcal{E})\in \mathcal{P}_\HH(\sigma)$. To ease notation, we will denote this set of geometric parameters as $\mathcal{P}_\HH(\lambda)$ in Section 4. 
 \end{rem}
 \begin{thm}{\label{H_quotient}}{\cite[Theorem 8.15]{Lusztig88}}
 Each simple $\HH$-module is isomorphic to the quotient of a standard $\HH$-module. 
 \end{thm}

 \begin{thm}{\cite[Corollary 8.18]{LLS2}}
The set of isomorphism classes of simple $\HH$-modules with central character $\chi_\lambda$ is naturally in 1-to-1 correspondence with $\mathcal{P}_\HH(\lambda)$. 
\end{thm}
 We will therefore use the notation $\irrH(\mathcal{O},\mathcal{E})$ to denote the irreducible module corresponding to parameter $(\mathcal{O},\mathcal{E})\in\mathcal{P}_\HH(\sigma)$.

\end{subsubsection}
\begin{subsubsection}{Algebraic classification}
Consider a subset $\Pi_\p$ of $\Pi$, and the corresponding roots (coroots) $R_\p$ (resp. $R^\vee_\p$) generated by $\alpha$ (resp. $\alpha^\vee$) for $\alpha\in \Pi_\p$. Then $(X,R_\p,Y,R^\vee_\p,\Pi_\p)$ is a root system. Let $\HH_\p$ be the graded affine Hecke algebra associated to the root system $(X,R_\p,Y,R^\vee_\p,\Pi_\p)$. Let $\mathfrak{a}$ be as in 3.1, and $\HH_{\s}$ denote corresponding subalgebra in the decomposition 
\beq
\HH_\p=\HH_\s\otimes S(\mathfrak{a}).
\eeq
\begin{thm}{\label{evens2.1}}{\cite[Theorem 2.1]{E}}
\begin{enumerate}[(i)]
\item Let $V$ be an irreducible $\HH$-module. Then $V$ is a quotient of $\HH\otimes_{\HH_\p}U$, where $U=\tilde{U}\otimes \mathbb{C}_\nu$ is such that $\tilde{U}$ is a tempered $\HH_{\s}$-module and $\nu\in\mathfrak{a}^\ast$ with $\text{Re}\langle\nu,\alpha\rangle>0$ for all $\alpha\in \Pi-\Pi_\p$. We will refer to $\HH\otimes_{\HH_\p}U$ as a standard module, and denote it by $\stdH(\HH_\p,U)$. 
\item If $U$ is as in (i), then $\HH\otimes_{\HH_\p}U$ has a unique irreducible quotient, which we will denote by $\irrH(\HH_\p,U)$. 
\item If $\irrH(\HH_\p,U)\cong \irrH(\HH_{\p'},U')$, then $\p=\p'$ and $U\cong U'$. 
 \end{enumerate}
\end{thm}
\end{subsubsection}

\begin{subsubsection}{Combinatorial classification}\label{multiseg-sec}

Let $\g=\mathfrak{sl}_\ell(\mathbb{C})$. For convenience, let $\h$ be the Cartan subalgebra consisting of diagonal matrices in $\g$, and $\bb$ be the Borel subalgebra consisting of upper triangular matrices in $\g$. Let $\HH$ be the graded affine Hecke algebra of the root system associated with $(\g,\bb)$. 
\begin{rem}
 When we want to emphasize that we are considering the graded affine Hecke algebra corresponding to the root datum of $\mathfrak{sl}_k(\mathbb{C})$ for some $k$, we will use the notation $\HH_k$. 
 \end{rem}
  Finite-dimensional irreducible $\HH$-modules are parametrized by combinatorial objects which we will refer to as multisegments. Define a segment to be a finite increasing sequence of complex numbers, such that any two consecutive terms differ by 1. A multisegment is an ordered collection of segments. Define the support $\underline{\tau}$ of a multisegment $\tau$ to be the multiset of all elements (so as to keep track of multiplicity) of all segments of the multisegment $\tau$. Since we are considering $\mathfrak{sl}_\ell$, we will only require multisegments with zero trace. More precisely, let MS be the set of multisegments $\tau$ such that $\sum_{x\in\underline{\tau}}x=0$. For $\lambda\in\h^\ast$, let $\lambda'\in\h$ be the image of $\lambda$ when we identify $\h^\ast$ with $\h$ using the trace form. For $\lambda\in \h^\ast$, set
\beq
\text{MS}(\lambda)=\{\tau\in\text{MS} | \underline{\tau}=\underline{\lambda}\},
\eeq
where we view $\underline{\lambda}$ as a multiset consisting of the diagonal entries of $\lambda'\in\h$. If $\tau,\sigma\in MS(\lambda)$, then define an equivalence relation by $\tau\sim \sigma$ if $\tau$ and $\sigma$ have the same segments (with a possibly different ordering). Let $MS_\circ(\lambda)$ denote the set of equivalence classes of $MS(\lambda)$. We will proceed by building a standard object in $\HH$ from a class of multisegments $\tilde{\tau}\in MS_\circ(\lambda)$. Let $\tilde{\tau}$ be represented by $\tau=\{\{a_1,a_1+1,a_1+2,\cdots,a_1+(l_1-1)\},\cdots,\{a_r,a_r+1,a_r+2,\cdots,a_r+(l_r-1)\}\}\in MS(\lambda)$ where $\text{Re}(a_i+\frac{1}{2}(l_i-1))\ge  \text{Re}(a_{i+1}+\frac{1}{2}(l_{i+1}-1))$ for all $i$.  Consider $\mathfrak{sl}(l_1)\oplus\cdots\oplus \mathfrak{sl}(l_r)$ as a block diagonal subalgebra of $\g$. Let $R_\lev$, $R^\vee_\lev$, and $\Pi_\lev$ be the set of roots, coroots, and simple roots for $\mathfrak{sl}(l_1)\oplus\cdots\oplus \mathfrak{sl}(l_r)$, respectively, chosen so that $\Pi_\lev\subset \Pi$. Then the graded affine Hecke algebra (denoted $\HH_\p$) associated with the root datum $(X,R_\lev,Y,R^\vee_\lev,\Pi_\lev)$ decomposes as 
\beq
\HH_\p=\HH_{ss}\otimes S(\mathfrak{a}),
\eeq
where $\HH_{ss}\cong \HH_{l_1}\otimes\cdots\otimes \HH_{l_r}$ is isomorphic to the graded affine Hecke algebra of the root datum associated with $\mathfrak{sl}(l_1)\oplus\cdots\oplus \mathfrak{sl}(l_r)$ and $S(\mathfrak{a})$ is as in Section 3.1. \\

Let $\gamma_i=\sum_{k=1}^{l_i} (a_i+k-1)\epsilon_k\in \td^\ast$ be viewed as an element in $\h_{l_i}^\ast$, the dual of the Cartan subalgebra of diagonal matrices intersected with the $\mathfrak{sl}(l_i)$ block of $\g$. Define the discrete series representation $\delta_{\tilde{\tau}}$ of $\HH_{ss}$ to be
\beq
\delta_{\tilde{\tau}}=\delta_1\boxtimes\cdots\boxtimes \delta_r,
\eeq
where $\delta_i= \mathbb{C}_{\gamma_i}$ is the one dimensional representation of $\HH_{l_i}$ (the graded affine Hecke algebra of the root datum associated with the algebra $\mathfrak{sl}(l_i)$) where $\h_{l_i}$ acts by weight $\gamma_i$ and $W_{l_i}$ acts by the sign representation. We will denote the standard $\HH_\ell$-module corresponding to $\tilde{\tau}$ by
\beq
\stdH(\tilde{\tau})=\HH_\ell\otimes_{\HH_\p}(\delta_{\tilde{\tau}}\boxtimes \mathbb{C}_\gamma),
\eeq
where $\mathbb{C}_\gamma$ is the character of $S(\mathfrak{a})$ given by $\gamma=\sum\gamma_i\in\h^\ast$ restricted to $\mathfrak{a}$. Notice that the module $\stdH(\tilde{\tau})$ is generated as a $\mathbb{C}[W]$-module by the vector 
\beq
\mathds{1}=e\otimes 1_1\otimes\cdots\otimes 1_r\otimes 1_\nu,
\eeq 
where $e\in \HH$ is the identity element, $1_i\in \delta_i$ is the identity in $\mathbb{C}_{\gamma_i}$, and $1_\nu$ is the identity in $\mathbb{C}_\nu$. Additionally $\mathds{1}$ is an $\h$-weight vector with weight $\zeta_{\tilde{\tau}}$ given by
\beq
\zeta_{\tilde{\tau}}(e_j^\vee)=a_i+j-\sum_{k=1}^{i-1}l_k-1\quad\text{ for }\sum_{k=1}^{i-1}l_k<j\le \sum_{k=1}^il_k.
\eeq
Here it is notationally easiest to define $\zeta_{\tilde{\tau}}$ as an element of $\td^\ast$, but we will only consider the restriction of $\zeta_{\tilde{\tau}}$ to $\h$. 
\begin{prop}
The standard $\HH_\ell$-module $\stdH(\tilde{\tau})$ has a unique simple quotient denoted $\irrH(\tilde{\tau})$.
\end{prop}
\begin{proof}
This follows directly from Theorem \ref{evens2.1}, which states that $\stdH(\tilde{\tau})$ has a unique irreducible quotient if $\gamma=\sum\gamma_i$ satisfies 
\beq
\text{Re}(\gamma(\alpha^\vee))\ge 0\quad\forall \alpha\in \Pi-\Pi_\lev.
\eeq
This is guaranteed by our choice of multisegment representative of $\tilde{\tau}$ chose to satisfy the condition 
\beq
\text{Re}(a_i+\frac{1}{2}\big(l_i-1)\big)\ge \text{Re}(a_{i+1}+\frac{1}{2}\big(l_{i+1}-1)\big),
\eeq
for all $i$. 
\end{proof}
\begin{thm}{\cite{BZ}}
 Suppose that $\lambda\in \h^\ast$ corresponds to $\sigma\in \h$ by the trace form on $\g$. There is a one-to-one correspondence between multisegments $\text{MS}_\circ(\lambda)$ and geometric parameters $\mathcal{P}_\HH(\sigma)$. 
\end{thm}

Following~\cite{AS}, we will now construct an $\HH_\ell$-module from a pair $\lambda,\mu\in\h^\ast$ with $\lambda-\mu\in P(V^{\otimes \ell})=\{\gamma\in\h^\ast: (V^{\otimes \ell})_\gamma\neq 0\}$. There exists $(\ell_1,\cdots,\ell_n)\in\mathbb{Z}^n_{\ge 0}$ so that $\ell=\sum_i\ell_i$ and 
\begin{equation}
\lambda-\mu\equiv \sum_{i=1}^n\ell_i\e_i\text{ mod}\mathbb{C}\bigg(\sum_{i=1}^n\e_i\bigg)
\end{equation}
Using the pair $\lambda,\mu\in\h^\ast$, with $\lambda$ dominant, we define the following multisegment
\beq
\delta_{\lambda,\mu}=\{\{(\mu+\rho)(\e_1^\vee),\cdots,(\mu+\rho)(\e_1^\vee)+\ell_1-1\},\cdots,\{(\mu+\rho)(\e_n^\vee),\cdots,(\mu+\rho)(\e_n^\vee)+\ell_n-1\}\}
\eeq
and set $\stdH(\lambda,\mu):=\stdH(\tilde{\delta}_{\lambda,\mu})$. The standard module $\stdH(\lambda,\mu)$ is a cyclic module with a cyclic weight vector $\mathds{1}$, whose weight $\zeta_{\lambda,\mu}$ is given by
\beqn
\zeta_{\lambda,\mu}(\e_j^\vee)=(\mu+\rho)(\e_i^\vee)+j-\sum_{r=1}^{i-1}\ell_r-1\quad\text{ for }\sum_{r=1}^{i-1}\ell_r<j\le\sum_{r=1}^{i}\ell_r.
\eeqn
Since $\lambda$ is dominant, $\stdH(\lambda,\mu)$ has a unique simple quotient denoted $\irrH(\lambda,\mu)$. 
 \begin{lem}{\label{dim_eq}}{\cite[Lemma 3.3.2]{AS}}
Let $P(V^{\otimes \ell})$ be the set of nontrivial weights of $V^{\otimes \ell}$. If $\lambda$ and $\mu$ are integral weights such that $\lambda-\mu\in P(V^{\otimes \ell})$, then 
 \beq
 \text{dim }\stdH(\lambda,\mu)=\text{dim}\big(V^{\otimes \ell}\big)_{\lambda-\mu}.
 \eeq
 \end{lem}

\begin{rem}
Notice that the multisegment $\delta_{\lambda,w\circ\lambda}$ is an element of $MS(\lambda+\rho)$, and the $\HH_\ell$-module $\stdH(\lambda,w\bullet\lambda)$ has central character $\chi_{\lambda+\rho}$. 
\end{rem}
\end{subsubsection}
\end{subsection}
\begin{subsection}{The composition series of standard modules}
We will now review the $p$-adic analogue of the Kazhdan-Lusztig conjectures. Using Lusztig's geometric realization of $\HH$-modules (\cite{LLS2}), we can relate the multiplicity of simple $\HH$-modules in the composition series of standard $\HH$-modules to stalks of intersection cohomology complexes. 
\begin{thm}{\label{geo_mult_hecke}}{\cite[Corollary 10.7]{LLS2}}
Suppose we have two geometric parameters $(\mathcal{O},\mathcal{E}),(\mathcal{O}',\mathcal{E}')\in\mathcal{P}_\HH(\sigma)$ such that $\mathcal{O}\subset\overline{\mathcal{O}'}$. Let $\text{IC}_y(\mathcal{O}',\mathcal{E}')$ denote the stalk at $y\in\mathcal{O}\subset\overline{\mathcal{O}'}$ of the intersection cohomology complex on $\mathcal{O}'$ corresponding to the local system $\mathcal{E}'$. Let $H^i(\text{IC}_y(\mathcal{O}',\mathcal{E}'))_\mathcal{E}$ denote the $\mathcal{E}$ isotypic component of $H^i(\text{IC}_y(\mathcal{O}',\mathcal{E}'))$, where we view $\mathcal{E}$ as a representation of the component group of $Z_G(\sigma,x)$ (where $\sigma$ and $x$ are as in Section \ref{Hecke-geo-class}). Then
\beq
[\stdH(\mathcal{O},\mathcal{E}):\irrH(\mathcal{O}',\mathcal{E}')]=\sum_i\text{dim}H^i(\text{IC}_y(\mathcal{O}',\mathcal{E}'))_\mathcal{E}.
\eeq
\end{thm}
When $\HH$ is a graded affine Hecke algebra corresponding to a root datum of type $A_n$, Theorem \ref{geo_mult_hecke} can be reformulated in terms of the combinatorial parametrization of \cite{BZ}, which proves a conjecture of Zelevinsky~\cite{Z}. We will thus proceed with the notation of Section 3.2.3. Suppose $\tau\in \text{MS}_\circ(\lambda)$ is a multisegment consisting of segments of size $l_1$ through $l_r$. Let $\sigma\in\h$ correspond to $\lambda\in \h^\ast$ by the trace form. Let $x_\tau\in\g_1(\sigma)$ denote the nilpotent matrix in $\g$ with a nonzero entry in the $i^{\text{th}}$ row and $i+1^{\text{th}}$ column for each $\sum_{j=1}^kl_j\le i< \sum_{j=1}^{k+1}l_{j}$, and zero entries elsewhere. Let $X_\tau$ be the $L_\sigma $ orbit of $x_\tau$. This gives us a bijection between $\text{MS}_\circ(\lambda)$ and $L_\sigma $ orbits on $\g_1(\sigma)$ (cf. \cite{Z81},\cite[Equation 4.2]{CT})
\beqn
\text{MS}_\circ(\lambda)&\longleftrightarrow& L_\sigma\backslash \g_1(\sigma)\\
\tau&\mapsto& X_\tau
\eeqn
For $G=SL_n$, we have that the component group of $Z_G(\sigma, x)$ is nontrivial. However, the only irreducible $L_\sigma$-equivariant local system which appears in the decomposition of $\mu_\ast^\sigma\mathcal{C}_{\widetilde{N}^\sigma}$ is the trivial local system. So the multisegment $\tau\in \text{MS}_\circ(\lambda)$ corresponds to geometric parameter $(X_\tau, \mathbb{C}_{X_\tau})\in\mathcal{P}_\HH(\sigma)$. 
\begin{cor}
For $\tau,\gamma\in \text{MS}_\circ(\lambda)$, we have that
\beq
[\stdH(\tau):\irrH(\gamma)]=\sum_{i\ge0}\text{dim}H^i(\text{IC}_y(X_\gamma,\mathbb{C}_{X_\gamma})),
\eeq
for $y\in {X_\tau}$.

\end{cor}


\end{subsection}

\begin{subsection}{$\HH_\ell$ action on $X\otimes V^{\otimes \ell}$}\label{hact}
To begin, let $B=\{E_i\}$ ($B^\ast=\{E_i^\ast\}$)  be an orthonormal basis (dual basis) of $\g$ with respect to the trace form. For notational convenience, let
\beqn
\tau_r(x)&=&1^{\otimes r-1}\otimes x\otimes 1^{\otimes \ell-r+1}\in U(\g)^{\otimes \ell+1},\\
\tau_{r,s}(y,z)&=&1^{\otimes r-1}\otimes  y\otimes 1^{\otimes s-r-1}\otimes z\otimes 1^{\otimes \ell-s+1} \in U(\g)^{\otimes\ell+1}.
\eeqn
We will define an operator $\Omega_{i,j}\in\text{End}(X\otimes V^{\otimes \ell})$ by
\beq
\Omega_{i,j}:=\sum_{E\in B}\tau_{i,j}(E,E^\ast).
\eeq
Consider the map $\Theta$ from $\HH_\ell$ to End$(X\otimes V^{\otimes\ell})$ defined by
\beqn
\Theta(s_i)&=&-\Omega_{i,i+1}\quad\text{ for $1\le i<\ell$},\\
\Theta(\e_k)&=&\frac{n-1}{2}+\sum_{0\le j<k}\Omega_{j,k}\quad\text{ for $1\le k\le \ell$},
\eeqn
where $\e_k$ is as in Section 3.1 and $s_i$ is the simple reflection defined by $s_i(\e_i)=\e_{i+1}$. 
\begin{lem}{\cite[Lemma 2.2.1]{AS}}
As operators on $X\otimes V^{\otimes \ell}$, we have the following equalities
\beqn
[\Theta(\e_i),\Theta(\e_j)]&=&0\\
\Theta(s_i)\Theta(\e_j)-\Theta(s_i(\e_j))\Theta(s_i)&=&\alpha_i(\e_j^\vee)\quad\text{ for $1\le i<\ell$ and $1\le j\le\ell$}.
\eeqn
\end{lem}

\begin{prop}{\label{commutes}}{\cite[Lemma 2.2.3]{CTU}}
For any $U(\g)$-module $X$, $X\otimes V^{\otimes \ell}$ is an $\HH_\ell$-module. Moreover, the action of $\HH_\ell$ commutes with the action of $\g$. 
\end{prop}

\end{subsection}

\end{section}

\begin{section}{Main Results}\label{sec-main}
In this section we will illuminate the relationship between irreducible Whittaker modules and irreducible $\HH$-modules for type $A_n$. Theorem \ref{aff_to_proj} shows that there is a natural correspondence between our parametrizations of irreducible objects in each category. This correspondence allows us to compare the respective character formulas, given by the Kazhdan-Lusztig polynomials. Theorem \ref{mult_equal}, which is the geometric core of our main result, shows that the multiplicity formulas agree under the correspondence in Theorem \ref{aff_to_proj}. In Section \ref{ASO}, we review the construction of Arakawa-Suzuki functors for the category of highest weight modules~\cite{AS}. We view these results as a categorical realization of the geometric results of Section \ref{geo_sec}. We conclude with our main result, the construction of Arakawa-Suzuki functors for Whittaker modules. 
\begin{subsection}{Equivariant maps between geometric parameters}\label{geo_sec}

Let $V=\mathbb{C}^n$, $G=\text{SL}(V)$, and $\g=\text{Lie}(G)$. In this section we will fix $\lambda\in\h^\ast$, and let $\sigma\in\h$ correspond to $\lambda+\rho$ by the trace form. Let $\g_i$ denote the $i$-eigenspace of ad$(\sigma)$. Let
\beq
\lev=\g_0,\qquad \uu=\bigoplus_{i>0}\g_i,\qquad \p=\lev\oplus\uu.
\eeq
There is a decomposition $V=V_1\oplus\cdots\oplus V_k$ such that 
\beqn
\lev&=&\{x\in\g|x(V_i)\subset V_i\},\\
\uu&=&\{x\in\g|x(V_i)\subset\bigoplus_{j<i}V_j\}.
\eeqn
Finally, let $L\subset G$ and $P\subset G$ be such that Lie$(L)=\lev$ and Lie$(P)=\p$. Consider the conjugation action of $L$ on $\g_{1}$. An orbit of $L$ on $\g_{1}$ is called a \emph{graded nilpotent class}~\cite{Z}. 

\begin{thm}{\cite[Theorem 1]{Z}}{\label{aff_to_proj}}
Let $F(V)$ be the variety of conjugates of $\p$. Let $N^t$ denote the transpose of $N$, and consider the following map  
\beqn
\varphi:\g_{1}&\rightarrow& F(V)\\
N&\mapsto&(1+N^t)\bullet \p:=(1+N^t)\p(1+N^t)^{-1}
\eeqn
We have that \begin{enumerate}[a)]
\item $\varphi$ is equivariant for the action of $L$.
\item $\varphi(L\bullet N)=L(1+N^t)\bullet\p$ is dense in $P(1+N^t)\bullet \p$.
\end{enumerate}
\end{thm}

Since $\varphi$ is $L$-equivariant, we can define an induced map on $L$ orbits of $\g_{1}$ as follows
\beqn
\Phi:L\backslash \g_{1}&\rightarrow & P \backslash G / P\\
Q&\mapsto & P\cdot\varphi(Q)
\eeqn 
Now we will relate the above map to the classifications of irreducible objects in Section 2 and Section 3. Recall the notation $\N(\xi(\lambda),\eta)$ denoting a subcategory of Whittaker modules from Section 2, and $\mathcal{P}_\HH(\lambda+\rho)$ denoting the geometric parametrization of irreducible modules for the graded affine Hecke algebra from Section 3. The double cosets $W_\eta\backslash W/W_\lambda$ parametrize the set of irreducible Whittaker modules in $\N(\xi(\lambda),\eta)$ (see Theorem \ref{irrN_class}). In the present setting, these double cosets are in one-to-one correspondence with $P_\eta$ orbits on $G/P$ (where $P_\eta$ is a Lie subgroup of $G$ with Lie algebra $\p_\eta$). Therefore, we denote the set of $P_\eta$ orbits on $G/P$ by $\mathcal{P}_\N(\eta,\lambda)$, and use this set to parametrize the irreducible objects in $\N(\xi(\lambda),\eta)$. Similarly, the geometric parameters $\mathcal{P}_\HH(\lambda+\rho)$ can be identified with the set of $L$ orbits on $\g_1$ (see the discussion following Theorem \ref{geo_mult_hecke}). If we assume that $W_\lambda=W_\eta$, then $P_\eta=P$, and we get the induced map on geometric parameters:
\beqn
\Psi:\mathcal{P}_\N(\eta,\lambda)&\rightarrow &\mathcal{P}_\HH(\lambda+\rho)\cup \{0\}\\
\mathcal{Q}&\mapsto & 
\begin{cases}
\Phi^{-1}(\mathcal{Q})\text{ if $\mathcal{Q}\in \text{im}\Phi$}\\
0\text{ otherwise}
\end{cases}
\eeqn
The above map of geometric parameters allows us to compare multiplicity formulas in each category using Corollary \ref{whit_geo_mult} and Theorem \ref{geo_mult_hecke}.

\begin{thm}{\label{mult_equal}}
Assume $\lambda$ is integral, dominant, and $W_\eta=W_\lambda$. Then for $\mathcal{Q},\mathcal{O}\in \mathcal{P}_\mathcal{N}(\eta,\lambda)$, we have the equality
\beq
[\stdN(\mathcal{Q}):\irrN(\mathcal{O})]=[\stdH(\Psi(\mathcal{Q})),\irrH(\Psi(\mathcal{O}))]
\eeq
if $\Psi(\mathcal{Q})$, $\Psi(\mathcal{O})\neq 0$. 
\end{thm}
\begin{proof}
Corollary \ref{whit_geo_mult} implies that
\beq
 [\stdN(\mathcal{Q}):\irrN(\mathcal{O})]=\sum_i\text{dim}H^i(\text{IC}_{x}(C({}_\eta y_\lambda))),
 \eeq
 where $C(w)$ is the open $B$ orbit in $\mathcal{Q}$, $x\in C(w)$, and $C({}_\eta y_\lambda)$ is the open $B$ orbit in $\mathcal{O}$.  Theorem \ref{geo_mult_hecke} implies that
 \beq
 [\stdH(\Psi(\mathcal{Q})):\irrH(\Psi(\mathcal{O}))]=\sum_i\text{dim}H^i(\text{IC}_{y}(\Psi(\mathcal{O})))
 \eeq
for $y\in\Psi(\mathcal{Q})$. Since the map $\varphi$ is stratum preserving, continuous, and has dense image in each stratum $S$ of $G/P$, we have that   $\text{dim}H^i(\text{IC}_{x}(C({}_\eta y_\lambda)))=\text{dim}H^i(\text{IC}_{y}(\Psi(\mathcal{O})))$ for each $i$ (\cite{Z},\cite[Section 4.8]{Kirwan}).
\end{proof}
Therefore, if there is an exact functor $F:\mathcal{N}(\xi(\lambda),\eta)\rightarrow \mathcal{H}_{\lambda+\rho}$ such that
 \beq
F(\stdN(\mathcal{O}))=\begin{cases}
\stdH(\Psi(\mathcal{O}))\text{ if }\Psi(\mathcal{O})\neq 0\\
0\text{ otherwise}
\end{cases}
\eeq
then we can conclude that
\beq
F(\irrN(\mathcal{O}))=
\begin{cases}
\irrH(\Psi(\mathcal{O}))
\text{ if $\Psi(\mathcal{O})\neq 0$}\\
0\text{ otherwise}
\end{cases}
\eeq
The content of the remainder of the paper will be focused on constructing such a functor for each choice of $\xi(\lambda)$ and $\eta$. 

\end{subsection}

\begin{subsection}{Arakawa-Suzuki functors for highest weight modules}\label{ASO}
In order to proceed constructing exact functors from the category of Whittaker modules to the category of graded affine Hecke algebra modules, we need to review some of the results of Arakawa-Suzuki. Specifically, following \cite{AS}, we can define a functor from $\mathcal{O}'$ to the category of vector spaces as follows:
\beq
F_{\ell,\lambda}(X):=\left(X\otimes V^{\otimes \ell}\right)^{[\lambda]}_\lambda.
\eeq
In this section we will briefly review results of \cite{AS} which show that $F_{\ell,\lambda}$ is a functor from $\mathcal{O}'$ to $\mathcal{H}_\ell$ (the category of finite dimensional $\HH_\ell$-modules) which maps standard modules in $\mathcal{O}'$ to standard modules in $\mathcal{H}_\ell$ (or zero). 

\begin{subsubsection}{Images of standard modules}

We will now review the following results of \cite{AS}. For a $U(\g)$-module $X$ where $\h$ acts semisimply, let $P(X):=\{\lambda\in\h^\ast: X_\lambda\neq 0\}$ be the set of all nontrivial weights of $X$. 
\begin{thm}{\label{AS}}{\cite[Theorem 3.3.1]{AS}}
For $\lambda,\mu$ integral weights with $\lambda$ dominant, there is an isomorphism of $\HH_\ell$-modules 
\beq
\left(M(\mu)\otimes V^{\otimes \ell}\right)^{[\lambda]}_\lambda\cong \begin{cases}
 \stdH(\lambda,\mu)\quad\text{ if $\lambda-\mu\in P(V^{\otimes \ell})$}\\
 0\qquad \qquad\text{ otherwise}
\end{cases}
\eeq 

\end{thm} 
In order to restate this theorem using the geometric parametrizations of standard and irreducible modules in the respective categories, we need to be working in the equal rank case, where $n=\ell$.
\begin{cor}
Suppose $\lambda$ is regular, integral, and dominant, and $\eta=0$. The geometric parameters for $\N(\xi(\lambda),\eta)$ are $B$ orbits on the flag variety $G/B$. Let $C(w)$ denote the $B$ orbit corresponding to $w\in W$. 
\beq
\left(M(\mu)\otimes V^{\otimes n}\right)^{[\lambda]}_\lambda=\begin{cases}
\stdH(\Psi(C(w)))\text{ if $\Psi(C(w))\neq 0$}\\
0\text{ otherwise}
\end{cases}
\eeq
Notice that $\Psi(C(w))\neq 0$ precisely when $\lambda-w\bullet\lambda\in P(V^{\otimes n})$. 
\end{cor}

\end{subsubsection}
\end{subsection}

\begin{subsection}{Arakawa-Suzuki functors for Whittaker modules}
For $\lambda \in\h^\ast$, we define the following functor from $\N(\eta)$ to the category of finite-dimensional vector spaces:
\begin{equation}
F_{\ell,\eta,\lambda}(X):=H^0_\eta\left(\n_\eta, (X\otimes V^{\otimes \ell})^{[\lambda]}_{\lambda_\z}\right).
\end{equation}
Since the action of $\HH_\ell$ on $X\otimes V^{\otimes \ell}$ commutes with the action of $\g$ (Proposition \ref{commutes}), $F_{\ell,\eta,\lambda}(X)$ has the structure of an $\HH_\ell$-module. 
\begin{prop}
For $\lambda$ dominant, $F_{\ell,\eta,\lambda}$ is an exact functor from $\N(\eta)$ to $\mathcal{H}$. 
\end{prop}
 \begin{proof}
 Following Proposition \ref{ex_inf}, Lemma \ref{exact}, and Proposition \ref{ex_wt}, we can see that $F_{\ell,\eta,\lambda}$ is exact when viewed as a functor from $\N(\eta)$ to the category of vector spaces. In order to see that $F_{\ell,\eta,\lambda}$ takes exact sequences of $U(\g)$-modules to exact sequences of $\HH_\ell$-modules (not just vector spaces), we will show that given a morphism of $U(\g)$-modules $\phi$, the linear map $F_{\ell,\eta,\lambda}(\phi)$ is actually a morphism of $\HH_\ell$-modules. Suppose $\phi:A\rightarrow B$ is a morphism of $U(\g)$-modules. Recall the operators $\Omega_{i,j}\in U(\g)^{\otimes\ell+1}$ from Section 3.4. Let $\phi':A\otimes V^{\otimes \ell}\rightarrow B\otimes V^{\otimes \ell}$ be defined by $\phi'(a\otimes v_1\otimes \cdots\otimes v_{\ell})=\phi(a)\otimes v_1\otimes \cdots\otimes v_{\ell}$. If $a\in A$, then $\Omega_{i,j}(\phi'(a\otimes v_1\otimes \cdots\otimes v_{\ell}))=\Omega_{i,j}(\phi(a)\otimes v_1\otimes \cdots\otimes v_{\ell})=\phi'(\Omega_{i,j}(a\otimes v_1\otimes \cdots\otimes v_{\ell}))$ then $F_{\ell,\eta,\lambda}(\phi)$. Since the action of $\HH_\ell$ is defined in terms of the operators $\Omega_{i,j}$, we have that $h\phi'(a\otimes v_1\otimes \cdots\otimes v_{\ell})=\phi'(h(a\otimes v_1\otimes \cdots\otimes v_{\ell}))$ for $h\in\HH_\ell$. It follows that $F_{\ell,\eta,\lambda}(\phi)$ is a morphism of $\HH_\ell$-modules. Hence $F_{\ell,\eta,\lambda}$ is an exact functor from $\N(\eta)$ to $\mathcal{H}$.
 \end{proof}
\begin{subsubsection}{Images of Standard Modules}

 In this section we will calculate the image of standard Whittaker modules. 
 
\begin{thm}{\label{induced_morph}}
The functor $\overline{\Gamma}_\eta$ induces a morphism of $\HH_\ell$-modules 
\beq
\overline{\Gamma}^\HH_\eta(\lambda,\mu):\text{Hom}_{U(\g)}\left(M(\lambda),M(\mu)\otimes V^{\otimes \ell}\right)\rightarrow H^0_\eta\left(\n_\eta, \left(\stdN(\mu,\eta)\otimes V^{\otimes \ell}\right)^{[\lambda]}_{\lambda_\z}\right).
\eeq
Moreover, $\overline{\Gamma}^\HH_\eta(\lambda,\mu)$ is an isomorphism if $\lambda$ is dominant and $W_\eta=W_\lambda$.
\end{thm}
\begin{proof}
We will begin by constructing a natural isomorphism 
 \beq
 \text{Hom}_{U(\g)}\left(\stdN(\lambda,\eta),\stdN(\mu,\eta)\otimes V^{\otimes \ell}\right)\cong H^0_\eta\left(\n_\eta, \left(\stdN(\mu,\eta)\otimes V^{\otimes \ell}\right)^{[\lambda]}_{\lambda_\z}\right).
 \eeq 
From Corollary \ref{dwt}, we see that $(\stdN(\mu,\eta)\otimes V^{\otimes \ell})^{[\lambda]}$ has a filtration with subquotients isomorphic to $\stdN(\lambda+\nu,\eta)$, where $\nu$ is a weight of $V^{\otimes \ell}$ and $\lambda+\nu\in W\bullet\lambda$.  Recall that $P(V^{\otimes n})$ denotes the set of nonzero $\h$-weights of $V^{\otimes n}$. By Theorem \ref{mcd}, we have that the nonzero $\z$-weight spaces of $\stdN(\lambda+\nu,\eta)$ have weights $w\bullet\lambda+\gamma$ for $\gamma\in P(U(\bar{\n}^\eta))$. Notice that $\mu\in \h^\ast$ and $w\bullet\mu$ are equal as $\z$-weights if $w\in W_\eta$. Since $\lambda$ is dominant, we have that $\stdN(w\bullet\lambda,\eta)$ has a nonzero $\z$-weight space of weight $\lambda$ if and only if $w\in W_\eta$. Moreover, since the action of $\n^\eta$ takes a $\z$-weight vector of weight $\lambda$ to a vector of weight $\lambda+\gamma'$ for $\gamma'\in P(\n^\eta)$, and $(w\bullet\lambda+P(U(\bar{\n}^\eta)))\cap( \lambda+P(\n^\eta))=\emptyset$, we can conclude that $\n^\eta v=0$ for all $\z$-weight vectors $v$ of weight $\lambda$ in $\stdN(w\bullet\lambda,\eta)$. In conclusion,  $v\in \stdN(\mu,\eta)\otimes V^{\otimes \ell}$ is contained in $ H^0_\eta\left(\n_\eta, \left(\stdN(\mu,\eta)\otimes V^{\otimes \ell}\right)^{[\lambda]}_{\lambda_\z}\right)$ if and only if $v$ is a $\z$-weight vector with weight $\lambda$, and has $xv=\eta(x)v$ for all $x\in \n$. Since $\stdN(\lambda,\eta)$ is isomorphic to $U(\g)v$ for each $v\in H^0_\eta\left(\n_\eta, \left(\stdN(\mu,\eta)\otimes V^{\otimes \ell}\right)^{[\lambda]}_{\lambda_\z}\right)$, we can define a morphism $\phi_v:\stdN(\lambda,\eta)\rightarrow\stdN(\mu,\eta)\otimes V^{\otimes \ell}$ for each $v\in H^0_\eta\left(\n_\eta, \left(\stdN(\mu,\eta)\otimes V^{\otimes \ell}\right)^{[\lambda]}_{\lambda_\z}\right)$. Alternatively, for each $\phi\in \text{Hom}_{U(\g)}(\stdN(\lambda,\eta),\stdN(\mu,\eta)\otimes V^{\otimes \ell})$, we can define a vector $v_\phi:=\phi(1\otimes 1\otimes 1)\in \stdN(\mu,\eta)\otimes V^{\otimes \ell}$ by considering the image of the canonical generator $1\otimes 1\otimes 1$ of $\stdN(\lambda,\eta)=U(\g)\otimes_{U(\p_\eta)}\big(U(\lev_\eta)/\xi_\eta(\lambda)U(\lev_\eta)
\otimes_{U(\n_\eta)}\mathbb{C}_\eta\big)$. Since $\phi$ is morphism of $U(\g)$-modules, it follows that $v_\phi$ is a $\z$-weight vector of weight $\lambda$ and $xv_\phi=\eta(x)v_\phi$ for all $x\in \n$. Moreover, the maps $v\mapsto \phi_v$ and $\phi\mapsto v_\phi$ are inverses of each other, which proves the natural identification described above.

Consider the map induced on morphisms by the functor $\overline{\Gamma}_\eta$. Let $\phi\in \text{Hom}_{U(\g)}(A, B)$ for $A,B$ in BGG category $\mathcal{O}$. Then $\overline{\Gamma}_\eta(\phi)\in \text{Hom}_{U(\g)}\left(\overline{\Gamma}_\eta(A),\overline{\Gamma}_\eta(B)\right)$ is given by
\beq
\overline{\Gamma}_\eta(\phi)\left(\sum_{\nu\in P(A)} x_\nu\right)=\sum_{\nu\in P(A)}\phi(x_\nu),
\eeq
where we view elements of $\overline{\Gamma}_\eta(A)$ (elements of $\overline{\Gamma}_\eta(B)$) as formal infinite linear sums of weight vectors of $A$ (weight vectors of $B$, respectively). Now, the action of $\HH_\ell$ on $\text{Hom}_{U(\g)}(M(\lambda),M(\mu)\otimes V^{\otimes \ell})$ (which we will denote here by $h.\phi$) is given by the action of $\HH_\ell$ on $M(\mu)\otimes V^{\otimes \ell}$. By Proposition \ref{ten_com}, we have that 
$$\text{Hom}_{U(\g)}\left(\overline{\Gamma}_\eta(M(\lambda)),\overline{\Gamma}_\eta(M(\mu)\otimes V^{\otimes\ell})\right)=\text{Hom}_{U(\g)}\left(\overline{\Gamma}_\eta(M(\lambda)),\overline{\Gamma}_\eta(M(\mu))\otimes V^{\otimes\ell}\right)$$
Since $\overline{\Gamma}_\eta(M(\lambda))$ is a $U(\g)$-module, we see that $\HH_\ell$ acts on $\text{Hom}_{U(\g)}\left(\overline{\Gamma}_\eta(M(\lambda)),\overline{\Gamma}_\eta(M(\mu))\otimes V^{\otimes\ell}\right)$ (which we will denote again by $h.\phi$). For $h\in\HH_\ell$ we have
\beq
h.\overline{\Gamma}_\eta(\phi)\left(\sum_{\nu\in P(M(\lambda))} x_\nu\right)=h\sum_{\nu\in P(M(\lambda))}\phi(x_\nu)=\sum_{\nu\in P(M(\lambda))}h\phi(x_\nu)=\overline{\Gamma}_\eta(h.\phi)\left(\sum_{\nu\in P(M(\lambda))} x_\nu\right).
\eeq
Therefore, $\overline{\Gamma}_\eta$ induces a morphism of $\HH_\ell$-modules. We will now turn our focus to showing that the induced morphism of $\HH_\ell$-modules is an isomorphism when $W_\eta=W_\lambda$, and $\mu$ is $\n_\eta$ antidominant. The generating vector $1\otimes 1 \otimes 1$ of $\stdN(\lambda,\eta)$ has $\z$-weight $\lambda$. Therefore, if we write the generating vector as a linear combination $\sum_{\nu\in P(M(\lambda))}x_\nu$, we have
\beq
z\sum_{\nu\in P(M(\lambda))}x_\nu=\sum_{\nu\in P(M(\lambda))}zx_\nu=\sum_{\nu\in P(M(\lambda))}\nu(z)x_\nu=\lambda(z)\sum_{\nu\in P(M(\lambda))}x_\nu,
\eeq
for all $z\in\z$. Therefore, the $\nu\in P(M(\lambda))$ for which $x_\nu$ is nonzero must have the property that $\nu(z)=\lambda(z)$ for all $z\in \z$. In other words, the $\nu\in P(M(\lambda))$ for which $x_\nu$ is nonzero must be contained in the set $P(U(\lev_\eta)v_\lambda)$, where $v_\lambda$ is a $\lambda$-highest weight vector in $M(\lambda)$. Since $W_\lambda=W_\eta$, we have that $U(\lev_\eta)v_\lambda$ is an irreducible $U(\lev_\eta)$-module. Therefore, if $\phi\in \text{Hom}_{U(\g)}(M(\lambda),M(\mu)\otimes V^{\otimes \ell})$ and $\phi(v_\lambda)\neq 0$, then $\phi(w)\neq 0$ for all (nonzero) $w\in U(\lev_\eta)v_\lambda$. Since $\phi$ is determined completely by its value on $v_\lambda$, we have that $\phi\neq 0$ implies that $\phi(v_\lambda)\neq 0$. This implies that $ \phi(x_\nu)\neq 0$ for each $x_\nu$ which appears in the sum decomposition of the generating vector of $\stdN(\lambda,\eta)$. Observe that since $\phi$ is a morphism of $U(\g)$-modules, it preserves weight spaces. Therefore, if $\nu_1\neq \nu_2\in\h^\ast$, we have $\phi(x_{\nu_1})+\phi(x_{\nu_2})=0$ if and only if $\phi(x_{\nu_1}),\phi(x_{\nu_2})=0$. So if $\phi\neq 0$, then 
\beq
\overline{\Gamma}_\eta(\phi)\left( \sum_{\nu\in P(M(\lambda))} x_\nu\right)=\sum_{\nu\in P(M(\lambda))}\phi(x_\nu)\neq 0.
\eeq
 We have therefore shown that the map $\overline{\Gamma}^\HH_\eta(\lambda,\mu)$ is injective. To conclude the proof, we will show that
\beq
\text{dimHom}_{U(\g)}(M(\lambda),M(\mu)\otimes V^{\otimes \ell})=\text{dimHom}_{U(\g)}\left(\stdN(\lambda,\eta),\stdN(\mu,\eta)\otimes V^{\otimes \ell}\right).
\eeq
Since $\lambda$ is dominant, $M(\lambda)$ is a projective object in $\mathcal{O}$, and the left hand side is equal to $[M(\mu)\otimes V^{\otimes \ell}:L(\lambda)]$, cf.~\cite{h}. Again, since $\lambda$ is dominant, $[M(w\bullet\lambda):L(\lambda)]$ is zero unless $w\bullet\lambda=\lambda$, in which case the multiplicity is 1. Therefore, the left hand side is equal to $\text{dim}(V^{\otimes \ell})_{\lambda-\mu}$ . On the right hand side, we have that $\stdN(\mu,\eta)\otimes V^{\otimes \ell}$ has a filtration with quotients isomorphic to $\stdN(\mu+\tau,\eta)$ for $\tau\in P(V^{\otimes \ell})$. So \beq 
\text{dimHom}_{U(\g)}(\stdN(\lambda,\eta),\stdN(\mu,\eta)\otimes V^{\otimes \ell})\le \sum_{\tau\in W_\eta\bullet\lambda-\mu} \text{dim}V^{\otimes \ell}_{\tau}.
\eeq
Since $W_\eta\bullet\lambda=\lambda$, we have that $\text{dimHom}_{U(\g)}(\stdN(\lambda,\eta),\stdN(\mu,\eta)\otimes V^{\otimes \ell})\le \text{dim}(V^{\otimes \ell})_{\lambda-\mu}$. Since $\overline{\Gamma}_\eta^\HH(\lambda,\mu)$ is injective, the inequality must be an equality, and the map must be an isomorphism. 
\end{proof}

\begin{thm}{\label{stds}}
Let $\lambda$ be dominant, integral, and suppose $W_\eta=W_\lambda$. Then
\beq
F_{\ell,\eta,\lambda}(\stdN(y\bullet\lambda,\eta))\cong\begin{cases}
 \stdH(\lambda, y\bullet\lambda)\quad\text{ if $\lambda-y\bullet\lambda\in P(V^{\otimes \ell})$}\\
 0\qquad\qquad\quad\hspace{15pt}\text{otherwise}\end{cases}
\eeq
Alternatively, under the geometric parametrization of standard objects (with the notation of Section 4), we have 
\beq
F_{\ell,\eta,\lambda}(\stdN(\mathcal{O}))\cong\begin{cases}
\stdH(\Psi(\mathcal{O}))\quad\text{ if $\Psi(\mathcal{O})\neq0$}\\
 0\qquad\qquad\quad\hspace{10pt}\text{otherwise}\end{cases}
\eeq
where $\mathcal{O}$ is a $P$ orbit on $G/P$ (recall that $P$ is a parabolic subgroup of $G$ whose Levi subgroup has Weyl group equal to $W_\lambda$). 
\end{thm}
\begin{proof}
By Theorem \ref{induced_morph},
\beq
\overline{\Gamma}^\HH_\eta(\lambda,\mu):\text{Hom}_{U(\g)}\left(M(\lambda),M(\mu)\otimes V^{\otimes \ell}\right)\rightarrow H^0_\eta\left(\n_\eta, \left(\stdN(\mu,\eta)\otimes V^{\otimes \ell}\right)^{[\lambda]}_{\lambda_\z}\right)
\eeq
is an isomorphism of $\HH_\ell$-modules. By Theorem \ref{AS},  
\beq
\left(M(\mu)\otimes V^{\otimes \ell}\right)^{[\lambda]}_\lambda\cong \begin{cases}
 \stdH(\lambda,\mu)\quad\text{if $\lambda-\mu\in P(V^{\otimes \ell})$}\\
 0\qquad \qquad\quad\text{otherwise}
\end{cases}
\eeq 
as $\HH_\ell$-modules. There is a canonical bijection $\text{Hom}_{U(\g)}\left(M(\lambda),M(\mu)\otimes V^{\otimes \ell}\right)\rightarrow \left(M(\mu)\otimes V^{\otimes \ell}\right)^{[\lambda]}_\lambda$~\cite[Remark 1.4.3]{AS}. Therefore 
\beqn
H^0_\eta\left(\n_\eta, \left(\stdN(\mu,\eta)\otimes V^{\otimes \ell}\right)^{[\lambda]}_{\lambda_\z}\right)&\cong& \text{Hom}_{U(\g)}\left(M(\lambda),M(\mu)\otimes V^{\otimes \ell}\right)\\
&\cong& \left(M(\mu)\otimes V^{\otimes \ell}\right)^{[\lambda]}_\lambda\\
&\cong& \begin{cases}
 \stdH(\lambda,\mu)\quad\text{if $\lambda-\mu\in P(V^{\otimes \ell})$}\\
 0\qquad \qquad\quad\text{otherwise}
\end{cases}
\eeqn
as $\HH_\ell$-modules. 
\end{proof}

\end{subsubsection}
\begin{subsubsection}{Images of Irreducible Modules}
Fix a character $\eta$ of $\n$, and let $\ell=n$. The following theorem is the main result of the paper, and gives an algebraic relationship between simple Whittaker modules and simple modules for the graded affine Hecke algebra. Moreover, this algebraic relationship agrees with the geometric relationship between the corresponding multiplicity formulas and intersection homologies observed by Zelevinsky~\cite{Z}. 
\begin{thm}{\label{main}}
Let $\lambda$ be integral and dominant. Assume $W_\lambda=W_\eta$. Let ${}_\eta y_\lambda$ denote the longest element in the double coset $W_\eta y W_\lambda$. If $\lambda-{}_\eta y_\lambda\bullet\lambda\in P(V^{\otimes n})$, then
\beq
F_{n,\eta,\lambda}\left(\irrN(y\bullet\lambda,\eta)\right)=\irrH(\lambda,{}_\eta y_\lambda\bullet\lambda).
\eeq
If $\lambda-{}_\eta y_\lambda\bullet\lambda\notin P(V^{\otimes n})$, then $F_{n,\eta,\lambda}(\irrN(y\bullet\lambda,\eta))=0$. Again, we will restate the theorem in the more natural setting of geometric parameters:
\beq
F_{n,\eta,\lambda}\left(\irrN(\mathcal{O})\right)\cong\begin{cases}
 \irrH(\Psi(\mathcal{O}))\quad\text{ if $\Psi(\mathcal{O})\neq0$}\\
 0\qquad\qquad\quad\hspace{5pt}\text{otherwise}\end{cases}
\eeq
\end{thm}
\begin{proof}
This follows directly from the statements following Theorem \ref{mult_equal}.    
\end{proof}
\begin{cor}
Every simple $\HH_n$-module appears as the image of an irreducible Whittaker module $\irrN(\mathcal{O})$ for some choice of $\eta\in\text{ch}\n$ and $\mathcal{O}$ a $P$ orbit on $G/P$. 
\end{cor}
\begin{proof}
For any give parameter $\mathcal{O}\in \mathcal{P}_\HH(\lambda+\rho)$, choose $\eta$ such that $W_\eta=W_\lambda$. Recall the map $\Phi$ from Section 4,
\beq
\Phi(\mathcal{O})=P\cdot (1+N^t)\cdot\p
\eeq
for some $N\in \mathcal{O}$. We have that $\Phi(\mathcal{O})$ is a single $P$ (hence $P_\eta$) orbit on $G / P$. Therefore $\Psi(\Phi(\mathcal{O}))=\mathcal{O}$, and 
\beq
F_{n,\eta,\lambda}(\irrN(\Phi(\mathcal{O})))=\irrH(\mathcal{O}).
\eeq
\end{proof}
\end{subsubsection}
\end{subsection}
\end{section}

\bibliographystyle{amsalpha}
\providecommand{\bysame}{\leavevmode\hbox to3em{\hrulefill}\thinspace}
\providecommand{\MR}{\relax\ifhmode\unskip\space\fi MR }
\providecommand{\MRhref}[2]{%
  \href{http://www.ams.org/mathscinet-getitem?mr=#1}{#2}
}
\providecommand{\href}[2]{#2}

\end{document}